\documentclass{pspum-l4arXiv}

\theoremstyle{definition}

\theoremstyle{remark}

\numberwithin{equation}{section}

\numberwithin{equation}{section}
\numberwithin{figure}{section}
  \theoremstyle{remark}
  \newtheorem*{rem*}{Remark}
\theoremstyle{plain}
\newtheorem{thm}{Theorem}[section]
  \theoremstyle{definition}
  \newtheorem{defn}[thm]{Definition}
  \theoremstyle{plain}
  \newtheorem{prop}[thm]{Proposition}
  \theoremstyle{plain}
  \newtheorem{lem}[thm]{Lemma}
  \theoremstyle{remark}
  \newtheorem{rem}[thm]{Remark}
 \theoremstyle{definition}
  \theoremstyle{plain}
  \newtheorem{fact}[thm]{Fact}
  \theoremstyle{plain}
  \newtheorem{conjecture}[thm]{Conjecture}
  \theoremstyle{plain}
  \newtheorem{cor}[thm]{Corollary}
  \theoremstyle{definition}
  \newtheorem{problem}[thm]{Problem}
  \newtheorem{example}[thm]{Example}

  \usepackage[all]{xy}
  \usepackage{amsthm}
\usepackage{amstext}
\usepackage{amssymb}

\makeatletter
\newcommand{\xyR}[1]{%
\makeatletter
\xydef@\xymatrixrowsep@{#1}
\makeatother
}  

\makeatletter
\newcommand{\xyC}[1]{%
\makeatletter
\xydef@\xymatrixcolsep@{#1}
\makeatother
}

\makeatother

\begin{document}

\title{From Operads to Dendroidal Sets}

\author{Ittay Weiss}
\address{Mathematical Institute, Utrecht University, Budapestlaan 6,
3584 CD, Utrecht, The Netherlands}
\email{I.Weiss@uu.nl}

\date{January 1, 1994 and, in revised form, June 22, 1994.}


\begin{abstract}
Dendroidal sets offer a formalism for the study of $\infty$-operads
akin to the formalism of $\infty$-categories by means of simplicial
sets. We present here an account of the current state of the theory
while placing it in the context of the ideas that led to the conception
of dendroidal sets. We briefly illustrate how the added flexibility
embodied in $\infty$-operads can be used in the study of $A_{\infty}$-spaces
and weak $n$-categories in a way that cannot be realized using strict
operads. 
\end{abstract}

\maketitle

\section{Introduction }

This work aims to be a conceptually self-contained introduction to
the theory and applications of dendroidal sets, surveying the current
state of the theory and weaving together ideas and results in topology
to form a guided tour that starts with Stasheff's work \cite{H spaces}
on $H$-spaces, goes on to Boardman and Vogt's work \cite{BV book}
on homotopy invariant algebraic structures followed by the generalization
\cite{ax hom th ope,BV res ope,res colour ope} of their work by Berger
and Moerdijk, arrives at the birth of dendroidal sets \cite{den set,inn Kan in dSet}
and ends with the establishment, by Cisinski and Moerdijk in \cite{dSet model hom op,den Seg sp,dSet and simp ope},
of dendroidal sets as models for homotopy operads. With this aim in
mind we adopt the convention of at most pointing out core arguments
of proofs rather than detailed proofs that can be found elsewhere. 

We assume basic familiarity with the language of category theory and mostly follow
\cite{Mac Lane CWM}. Regarding enriched category theory we assume
little more than familiarity with the definition of a category enriched
in a symmetric monoidal category as can be found in \cite{Kelly Enr. Cat. }.
The elementary results on presheaf categories that we use can be found in \cite{Mac Moer Sheaves}.
Some comfort of working with simplicial sets is needed for which the
first chapter of \cite{Goers Jardin} suffices. Some elementary understanding
of Quillen model categories is desirable with standard references
being \cite{model cat localiza,Hovey Model Cat}. 

Operads arose in algebraic topology and have since found applications
across a wide range of fields including Algebra, Theoretical Physics,
and Computer Science. The reason for the success of operads is that
they offer a computationally effective formalism for treating algebraic
structures of enormous complexity, usually involving some notion of
(abstract) homotopy. As such, the first operads to be introduced were
topological operads and most of the other variants are similarly enriched
in other categories. However, the presentation we give here of operads
treats them as a rather straightforward generalization of the notion
of category. It is that viewpoint that quite naturally leads to defining
dendroidal sets to serve as the codomain category for a nerve construction
for operads, extending the usual nerve of categories. 

The path we follow is the following one. We first examine the expressive
power of non-enriched symmetric operads. We find that by considering
operad maps it is possible to classify a wide range of strict algebraic
structures such as associative and commutative monoids and to show that operads
carry a closed monoidal structure that, via the internal hom, internalizes
algebraic structures. We show that the rather trivial fact that algebraic
structures can be transferred along isomorphisms, which we call the
isomorphism invariance property, is a result of symmetric operads
supporting a Quillen model structure compatible with the monoidal
structure. We then turn to the much more challenging homotopy invariance
property for algebraic structures in the presence of homotopy notions.
We show how the theory of operads is used to adequately handle this
more subtle situation, however at a cost. The internalization of these
so-called weak algebras, via an internal hom construction, is lost.
The sequel can be seen as a presentation of the successful attempt
to develop a formalism for weak algebraic structures in which the
internalization of algebras is restored. This formalism is given by
dendroidal sets and a suitable Quillen model structure which can be
used to give a proof of the homotopy invariance property (which is completely
analogous to the case of non-enriched operads). The consequences and
applicability of the added flexibility of dendroidal sets is portrayed
by considering $n$-fold $A_{\infty}$-spaces and weak $n$-categories. 

Section 2 introduces in the first half non-enriched symmetric operads
and presents their basic theory. The second half is concerned with
enriched operads and the Berger-Moerdijk generalization of the Boardman-Vogt
$W$-construction. Section 3 is a parallel development of the ideas
in Section 2. The first half introduces dendroidal sets and presents
their basic theory while the second half is concerned with the homotopy
coherent nerve construction with applications to $A_{\infty}$-spaces
and weak $n$-categories. Section 4 is devoted to the Cisinski-Moerdijk
model structure on dendroidal sets and the way it is used to prove
the homotopy invariance property. Section 5 closes this work with a brief presentation
of a planar dendroidal Dold-Kan correspondence and discusses the yet
unsolved problem of obtaining a satisfactory geometric realization
for dendroidal sets. 
\begin{rem*}
Below we work in a convenient category of topological spaces $Top$.
In some places it is important that this category be closed monoidal, in which
case the category of compactly generated Hausdorff spaces would suffice.
We will not remark about such issues further. 
\end{rem*}

\subsection*{Acknowledgements}
The author wishes to thank the referee for the diligent reading of an earlier version of this work and for the numerous helpful remarks and comments that helped shape the article. 

\section{Operads and algebraic structures}
\begin{rem*}
The reader already familiar with operads who reads this section just
to familiarize herself with the notation is strongly advised to look
at Remark \ref{rem:To-distinguish-between} and  Fact \ref{fac:iso invar}
below. For her convenience the opening paragraph below recounts the
contents of the entire section. 
\end{rem*}
Our journey starts with non-enriched symmetric operads, also known
as symmetric multicategories (originating in Lambek's study of deductive
systems in logic \cite{Lambek multi cat}) or symmetric coloured operads
(e.g., \cite{res colour ope,Leinster higher}). In the literature on operads
these structures are underrepresented probably due to the fact that
the first operads, introduced by May in \cite{May GILS}, were enriched
in topological spaces and many of the most important uses of operads
require enrichment. The point of view of operads we adopt is that
operads generalize categories. Consequently, just as a study of categories
starts with non-enriched categories, with enrichment usually treated
at some later stage, we first present non-enriched symmetric operads.
The operadically versed reader will immediately recognize that our
definition of algebra differs slightly from the standard one. We define the
Boardman-Vogt tensor product of symmetric operads and the notion of
natural transformations for symmetric operads that endows the category
of symmetric operads with the structure of a symmetric closed monoidal category.
We then address the isomorphism invariance property and treat it in
the context of a suitable Quillen model structure on symmetric operads.
We then turn to the much more subtle and interesting case of the homotopy
invariance property and give an expository treatment of the theory
developed by Berger and Moerdijk relevant for the rest of the presentation.

\subsection{Trees}

Symmetric (also called 'non-planar') rooted trees are useful in the
study of symmetric operads. There is no standard definition of 'tree'
that is commonly used (Ginzburg and Kapranov in \cite{GK Kozul duality ope}
use a topological definition while Leinster in \cite{Leinster higher}
uses a combinatorial one) but all approaches are essentially the same.
More recently, Joachim Kock in \cite{Kock poly func} established
a close connection between trees and polynomial functors, offering
yet another formalism of trees while shedding a different light on
the symbiosis between operads and trees. 

We present here the formalism of trees we use and introduce terminology
for commonly occurring trees as well as grafting of trees. We end
the section by presenting a generalization of posets that shows trees
to be analogues of finite linear orders.

\subsubsection{Symmetric rooted trees}
\begin{defn}
A \emph{tree} (short for symmetric rooted tree) is a finite poset
$(T,\le)$ which has a smallest element and such that for each $e\in T$
the set $\{y\in T\mid y\le e\}$ is linearly ordered.
The elements of $T$ are called \emph{edges }and the unique smallest
edge is called the \emph{root}. Part of the information of a tree
is a subset $L=L(T)$ of maximal elements, which are called \emph{leaves}. 
An edge is \emph{outer} if it is either the root or it belongs to $L$, otherwise
it is called \emph{inner}.
\end{defn}
Given edges $e,e'\in T$ we write $e/e'$ if $e'<e$ and if for any $x\in T$ for which $e'\le x\le e$
holds that either $x=e'$ or $x=e$ . For a non-leaf edge $e$ the set $in(e)=\{t\in T\mid t/e\}$
is called the set of \emph{incoming edges} into $e$. For such an edge $e$ the set $v=\{e\}\cup in(e)$
is called the \emph{vertex} above $e$ and we define $in(v)=in(e)$ and $out(v)=e$
which are called, respectively, the set of \emph{incoming edges} and
the \emph{outgoing edge} associated to $v$. The \emph{valence }of
$v$ is equal to $|in(v)|$ and could be $0$. Note that there is no vertex associated to a leaf. We will draw trees
by the graph dual of their Hesse diagrams with the root at the bottom
and will use a $\bullet$ for vertices. For example, in the tree
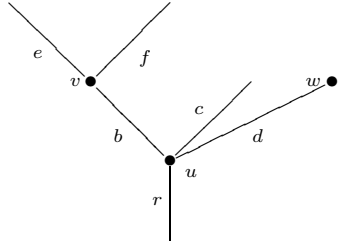
\[
\xymatrix{*{\,}\ar@{-}[dr]_{e} &  & *{\,}\ar@{-}[dl]^{f}\\
\,\ar@{}[r]|{\,\,\,\,\,\,\,\,\,\,\,\,\,\, v} & *{\bullet}\ar@{-}[dr]_{b} &  & *{\,}\ar@{-}[dl]_{c}\ar@{}[r]|{\,\,\,\,\,\,\,\,\,\,\,\, w} & *{\bullet}\ar@{-}[dll]^{d}\\
 &  & *{\bullet}\ar@{-}[d]_{r} & \,\ar@{}[l]^{u\,\,\,\,\,\,\,\,\,\,\,}\\
 &  & *{\,}}
\]
there are three vertices of valence 2,3, and 0 and three leaves $L=\{e,f,c\}$.
The outer edges are $e,f,c$, and $r$, where $r$ is the root. The
inner edges are then $b$ and $d$.

\subsubsection{Some common trees}

The following types of trees appear often enough in the theory of dendroidal
sets to merit their own notation. 
\begin{defn}
For each $n\ge 0$, a tree $L_{n}$ of the form\[
\xymatrix{*{}\ar@{-}[d]\\
*{\bullet}\ar@{..}[d]\\
*{\bullet}\ar@{-}[d]\\
*{\bullet}\ar@{-}[d]\\
*{}}
\]
with one leaf and $n$ vertices, all unary (i.e., each vertex has valence
equal to $1$), will be called a \emph{linear tree of order $n$}. The special case
of the tree $L_{0}$
\end{defn}
\[
\xymatrix{*{}\ar@{-}[d]\\
*{}}
\]
consisting of just one edge and no vertices is called the \emph{unit}
tree. We denote this tree by $\eta$, or $\eta_{e}$ if we wish to
explicitly name its unique edge. In this tree, the only edge is both
the root and a leaf. 
\begin{defn}
For each $n\ge 0$, a tree $C_{n}$ of the form 
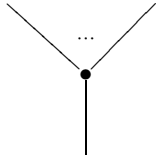
\[
\xymatrix{*{}\ar@{-}[dr] & *{}\ar@{}[d]|{\cdots} & *{}\ar@{-}[dl]\\
 & *{\bullet}\ar@{-}[d]\\
 & *{}}
\]
that has just one vertex and $n$ leaves will be called an \emph{$n$-corolla}.
Note that the case $n=0$ results in a tree different than $\eta$. 
\end{defn}

\subsubsection{Grafting}
\begin{defn}
Let $T$ and $S$ be two trees whose only common edge is the root
$r$ of $S$ which is also one of the leaves of $T$. The \emph{grafting},
$T\circ S$, of $S$ on $T$ along $r$ is the poset $T\cup S$ with
the obvious poset structure and set of leaves equal to $(L(S)\cup L(T))-\{r\}$. 
\end{defn}
Pictorially, the grafted tree $T\circ S$ is obtained by putting the
tree $S$ on top of the tree $T$ by identifying the output edge of
$S$ with the input edge $r$ of $T$. By repeatedly grafting, one
can define a full grafting operation $T\circ(S_{1},\cdots,S_{n})$
in the obvious way. 

We now state a useful decomposition of trees that allows for inductive
proofs on trees. The proof is trivial. 
\begin{prop}
Let $T$ be a tree. Suppose $T$ has root $r$ and $\{r,e_{1},\cdots,e_{n}\}$
is the vertex above $r$. Let $T_{e_{i}}$ be the tree that contains
the edge $e_{i}$ as root and everything above it in $T$. Then \[
T=T_{root}\circ(T_{e_{1}},\cdots,T_{e_{n}})\]
where $T_{root}$ is the n-corolla consisting of $r$ as root and
$\{e_{1},\cdots,e_{n}\}$ as the set of leaves.
\end{prop}

\subsubsection{Trees and dendroidally ordered sets}

The trees we defined above are going to be the objects of the category
$\Omega$ whose presheaf category $Set_{\Omega}$ is the category
of dendroidal sets. Recall that the simplicial category $\Delta$
(whose presheaf category is the category of simplicial sets) can be
defined as (a skeleton of) the category of totally ordered finite
sets with order preserving maps. In this section we present an extension
of the notion of totally ordered finite sets closely related to trees.
The content of this section is not used anywhere in the sequel and
is presented for the sake of completeness. Consequently we give no
proofs and refer the reader to \cite{thesis} for more details. 

First we extend the notion of a relation and that of a poset to what
we call broad relation and broad poset. For a set $A$ we denote by
$A^{+}=(A^{+},+,0)$ the free commutative monoid on $A$. A \emph{broad
relation} is a pair $(A,R)$ where $A$ is a set and $R$ is a subset
of $A\times A^{+}$. As is common with ordinary relations, we use
the notation $aR(a_{1}+\cdots+a_{n})$ instead of $(a,(a_{1}+\cdots+a_{n}))\in R$. 
\begin{defn}
A \emph{broad poset} is a broad relation $(A,R)$ satisfying:\end{defn}
\begin{enumerate}
\item Reflexivity: $aRa$ holds for any $a\in A$.
\item Transitivity: For all $a_{0},\cdots,a_{n}\in A$ and $b_{1},\cdots,b_{n}\in A^{+}$ such that $a_{i}Rb_{i}$ for $1\le i\le n$, holds that
if $a_{0}R(a_{1}+\cdots+a_{n})$ then $aR(b_{1}+\cdots+b_{n})$.
\item Anti-symmetry: For all $a_{1},a_{2}\in A$ and $b_{1},b_{2}\in A^{+}$
if $a_{1}R(a_{2}+b_{2})$ and $a_{2}R(a_{1}+b_{1})$ then $a_{1}=a_{2}$. 
\end{enumerate}
When $(A,R)$ is a broad poset we denote $R$ by $\le$. The meaning
of $<$ is then defined in the usual way. 

A \emph{map} of broad posets $f:A\rightarrow B$ is a set function
preserving the broad poset structure, that is if $a\le(a_{1}+\cdots+a_{n})$
then $f(a)\le(f(a_{1})+\cdots+f(a_{n}))$. 
\begin{defn}
We denote by $BrdPoset$ the category of all broad posets and their
maps. 
\end{defn}
Let $\star$ be a singleton set $\{*\}$ with the broad poset structure
given by $*\le*$. Note that $\star$ is not a terminal object in
$BrdPoset$. 
\begin{lem}
(Slicing lemma for broad posets) There is an isomorphism of categories
between $BrdPoset/\star$ and the category $Poset$ of posets and
order preserving maps. Moreover, along this isomorphism one obtains
a functor $k_{!}:Poset\to BrdPoset$ which has a right adjoint $k^{*}:BrdPoset\to Poset$
which itself has a right adjoint $k_{*}:Poset\to BrdPoset$. 
\end{lem}
As motivation for the following definition recall that a finite ordinary
poset $A$ is linearly ordered if, and only if, it has a smallest
element and for every $a\in A$ the set $a_{\uparrow}=\{x\in A\mid a<x\}$
is either empty or has a smallest element. 
\begin{defn}
A finite broad poset $A$ is called \emph{dendroidally ordered} if
\end{defn}
\begin{enumerate}
\item There is an element $r\in A$ such that for every $a\in A$ there
is $b\in A^{+}$ such that $r\le a+b$. 
\item For every $a\in A$ the set $a_{\uparrow}=\{b\in A^{+}\mid a<b\}$
is either empty or it contains an element $s(a)=a_{1}+\cdots+a_{n}$
such that every $b\in a_{\uparrow}$ can be written as $b=b_{1}+\cdots+b_{n}$
with $a_{i}\le b_{i}$ for all $1\le i\le n$. 
\item For every $a_{0},\cdots,a_{n}\in A$, if $a_{0}\le a_{1}+\cdots+a_{n}$
then for $i\ne j$ there holds $a_{i}\ne a_{j}$. 
\end{enumerate}
Trees are related to finite dendroidally ordered sets as follows. Given a tree
$T$ define a broad relation on the set $E(T)$ of edges by declaring
$e\le e_{1}+\cdots+e_{n}$ precisely when there is a vertex $v$ such
that $in(v)=\{e_{1},\cdots,e_{n}\}$ (without repetitions) and $out(v)=e$.
The transitive closure of this broad relation is then a dendroidally
ordered set. This constructions can be used to give an equivalence
of categories between the full subcategory $DenOrd$ of $BrdPoset$
spanned by the dendroidally ordered sets and the dendroidal category $\Omega$
defined below. It is easily seen that $DenOrd$, upon slicing over
$\star$, is isomorphic to the category of all finite linearly ordered
sets and order preserving maps.

\subsection{Operads and algebras}

We now present symmetric operads viewed as a generalization of categories
where arrows are allowed to have domains of arity $n$ for any $n\in\mathbb{N}$.
We then define the notion of $\mathcal{P}$-algebras for a symmetric
operad $\mathcal{P}$ which are often referred to as the raison d'\^etre
of operads. We deviate here from the common definition of algebras
noting that our definition encompasses the standard one. We define
an algebra to simply be a morphism between symmetric operads, the
difference being purely syntactic. The assertion that symmetric operads
exist in order to define algebras thus agrees with the idea that in
any category the objects' raison d'\^etre is to serve as domains and
codomains of arrows. 
\begin{defn}
A \emph{planar operad} $\mathcal{P}$ consists of a class $\mathcal{P}_{0}$
whose elements are called the \emph{objects} of $\mathcal{P}$ and
to each sequence $P_{0},\cdots,P_{n}\in\mathcal{P}_{0}$ a set $\mathcal{P}(P_{1},\cdots,P_{n};P_{0})$
whose elements are called \emph{arrows} depicted by $\psi:(P_{1},\cdots,P_{n})\rightarrow P_{0}$\emph{.}
With this notation $(P_{1},\cdots,P_{n})$ is the \emph{domain} of
$\psi$, $P_{0}$ its \emph{codomain}, and $n$ its \emph{arity} (which
is allowed to be $0$). The domain and codomain are assumed to be
uniquely determined by $\psi$. There is for each object $P\in\mathcal{P}_{0}$
a chosen arrow $id_{P}:P\rightarrow P$ called the \emph{identity
}at $P$. There is a specified composition rule: Given $\psi_{i}:(P_{1}^{i},\cdots,P_{m_{i}}^{i})\rightarrow P_{i}$,
$1\le i\le n$, and an arrow $\psi:(P_{1},\cdots,P_{n})\rightarrow P_{0}$
their composition is denoted by $\psi\circ(\psi_{1},\cdots,\psi_{n})$
and has domain $(P_{1}^{1},\cdots,P_{m_{1}}^{1},P_{1}^{2},\cdots,P_{m_{2}}^{2},\cdots,P_{1}^{n},\cdots,P_{m_{n}}^{n})$
and codomain $P_{0}$. The composition is to obey the following unit
and associativity laws:\end{defn}
\begin{itemize}
\item Left unit axiom: $id_{P}\circ\psi=\psi$
\item Right unit axiom: $\psi\circ(id_{P_{1}},\cdots,id_{P_{n}})=\psi$
\item Associativity axiom: the composition \[
\psi\circ(\psi_{1}\circ(\psi_{1}^{1},\cdots,\psi_{m_{1}}^{1}),\cdots,\psi_{n}\circ(\psi_{1}^{n},\cdots,\psi_{m_{n}}^{n}))\]
is equal to\[
(\psi\circ(\psi_{1},\cdots,\psi_{n}))\circ(\psi_{1}^{1},\cdots,\psi_{m_{1}}^{1},\cdots,\psi_{1}^{n},\cdots,\psi_{m_{n}}^{n}).\]

\end{itemize}
The morphisms of planar operads are the obvious structure preserving
maps. A map of operads will also be referred to as a \emph{functor}. 
\begin{defn}
A \emph{symmetric operad} is a planar operad $\mathcal{P}$ together with actions of the symmetric groups in the following sense: for each $n\in\mathbb{N}$, objects $P_{0},\cdots,P_{n}\in\mathcal{P}_{0}$,
and a permutation $\sigma\in\Sigma_{n}$ a function $\sigma^{*}:\mathcal{P}(P_{1},\cdots,P_{n};P_{0})\to\mathcal{P}(P_{\sigma(1)},\cdots,P_{\sigma(n)};P_{0})$.
We write $\sigma^{*}(\psi)$ for the value of the action of $\sigma$ on  $\psi:(P_{1},\cdots,P_{n})\to P_{0}$ and demand that for any two permutations
$\sigma,\tau\in\Sigma_{n}$ there holds $(\sigma\tau)^{*}(\psi)=\tau^{*}\sigma^{*}(\psi)$.
Moreover, these actions of the permutation groups are to be compatible
with compositions in the obvious sense (see \cite{Leinster higher,May GILS}
for more details). \emph{Functors} of symmetric operads $\mathcal{P}\to\mathcal{Q}$
are functors of the underlying planar operads that respect the actions of the symmetric
groups.
\end{defn}
When dealing with operads we make a distinction between small and
large ones according to whether the class of objects is, respectively,
a set or a proper class. If more care is needed and size issues become
important we implicitly assume working in the formalism of Grothendieck
universes (\cite{Groth univers}) similarly to the way such issues
are avoided in category theory. We now obtain the category $Ope_{\pi}$
of small planar operads and their functors as well as the category
$Ope$ of small symmetric operads and their functors. There is clearly
a forgetful functor $Ope\to Ope_{\pi}$ which has an easily constructed
left adjoint $S:Ope_{\pi}\to Ope$ called the \emph{symmetrization
}functor. 
\begin{rem}
We note that our symmetric operads are also called symmetric multicategories
(see e.g., \cite{Leinster higher}) as well as symmetric coloured
operads. The composition as given above is sometimes called full $\circ$
composition. Using the identities in an operad we can then define what is known
as the $\circ_{i}$ composition as follows. Given an arrow $\psi$
of arity $n$ and $1\le i\le n$ one can compose an arrow $\varphi$
onto the $i$-th place of the domain of $\psi$, provided the object
at the $i$-th place is equal to the codomain of $\varphi$, by means
of $\psi\circ_{i}\varphi=\psi\circ(id,\cdots,id,\varphi,id,\cdots id)$
with $\varphi$ appearing in the $i$-th place. Some authors consider
operads defined in terms of the $\circ_{i}$ operations rather than
the full $\circ$ composition. In the presence of identities there
is no essential difference but if identities are not assumed than
one obtains a slightly weaker structure called a \emph{pseudo-operad
}(see \cite{MSS book})\emph{.} The operads we consider always have
identities so that the full $\circ$ and partial $\circ_{i}$ compositions
differ only cosmetically and will be used interchangeably as convenient. 
\end{rem}
Operads are closely related to categories. Indeed, one trivially sees
that a category is an operad where each arrow has arity equal
to $1$. 

A slightly less trivial and more useful fact is the following. Call a symmetric operad \emph{reduced} if it has no $0$-ary operations. We
denote by $\star$ an operad with one object and only the identity
arrow on it. 
\begin{lem}
(Slicing lemma for symmetric operads) There is an isomorphism between
the category $Cat$ of small categories and the slice category $Ope/\star$.
Moreover, there are functors $j_{!}:Cat\to Ope$ and $j^{*}:Ope\to Cat$
such that $j^{*}$ is right adjoint to $j_{!}$. The functor $j^{*}$ does not preserve pushouts and thus does not have a right adjoint. However, the restriction of $j_{*}$ to the subcategory of reduced operads does have a right adjoint. Under the isomorphism $Cat\cong Ope/\star$ the
functor $j_{!}$ is the forgetful functor $Cat=Ope/\star\to Ope$. \end{lem}
\begin{proof}
We explicitly describe the functors, omitting any details. Given a
category $\mathcal{C}$ the operad $j_{!}(\mathcal{C})$ has $j_{!}(\mathcal{C})_{0}=\mathcal{C}_{0}$ (here $\mathcal{C}_{0}$ stands for the class of objects of the category $\mathcal{C}$)
and the arrows in $j_{!}(\mathcal{C})$ are given for $P_{0},\cdots,P_{n}\in j_{!}(\mathcal{C})_{0}$
as follows:\[
j_{!}(\mathcal{C})(P_{1},\cdots,P_{n};P_{0})=\begin{cases}
\mathcal{C}(P_{1},P_{0}) & \mbox{if }n=1\\
\emptyset & \mbox{if }n\ne1\end{cases}\]
The composition is the same as in $\mathcal{C}$. The right adjoint
$j^{*}$ is given for an operad $\mathcal{P}$ as follows. $j^{*}(\mathcal{P})_{0}=\mathcal{P}_{0}$
and the arrows in $j^{*}(\mathcal{P})$ are given for $C,D\in j^{*}(\mathcal{P})_{0}$
by:\[
j^{*}(\mathcal{P})(C,D)=\mathcal{P}(C;D).\]
The composition is the same as in $\mathcal{P}$. Finally, the functor
$j_{*}$, right adjoint to the restriction of $j^{*}$ to reduced operads, is defined for a category $\mathcal{C}$ as follows.
$j_{*}(\mathcal{C})_{0}=\mathcal{C}_{0}$ and the arrows in $j_{*}(\mathcal{C})$
are given for $P_{0},\cdots,P_{n}\in j_{*}(\mathcal{C})$ as follows:\[
j_{*}(\mathcal{C})(P_{1},\cdots,P_{n};P_{0})=\begin{cases}
\mathcal{C}(P_{1},P_{0}) & \mbox{if }n=1\\
\{(P_{1},\cdots,P_{n};P_{0})\} & \text{if }n\ne1\end{cases}\]
Composition of unary arrows is given as in $\mathcal{C}$. Composition
of two arrows at least one of which is not unary is uniquely determined
since the hom set of where that arrow is to be found consists of just
one object. It is therefore automatic that the composition so defined
is associative. \end{proof}
\begin{rem}
\label{rem:Slicing}The construction of the three functors above follows
from general abstract nonsense and is related to locally cartesian
closed categories. Indeed, if  $\mathcal{C}$ is a category with a terminal
object $*$, then for any object $A\in\mathcal{C}_{0}$ the unique arrow
$A\to*$ induces a functor between the slice categories $F_{!}:\mathcal{C}/A\to\mathcal{C}/*$.
It is then a general result that $F_{!}$ has a right adjoint $F^{*}$
if, and only if, $\mathcal{C}$ admits products with $A$. Moreover,
$F^{*}$ has a right adjoint $F_{*}$ if, and only if, $A$ is exponentiable
in $\mathcal{C}$. The case we had at hand is when $\mathcal{C}$ is the category of symmetric operads, or its subcategory of reduced symmetric operads, 
and $A=\star$. 
\end{rem}
Due to this intimate connection between symmetric operads and categories
we will employ category theoretic terminology in the context of symmetric
operads. For example, we will refer to morphisms of operads as functors,
and feel free to use category theoretic terminology within the 'category
part' $j^{*}(\mathcal{P})$ of an operad $\mathcal{P}$. So the notion
of a unary arrow $f$ in $\mathcal{P}$ being, for instance, an isomorphism,
a monomorphism, or a split idempotent simply means that $f$ has the
same property in the category $j^{*}(\mathcal{P})$. In this spirit
we give the following definition of equivalence of operads. 
\begin{defn}
Let $\mathcal{P}$ and $\mathcal{Q}$ be symmetric operads and $F:\mathcal{P}\to\mathcal{Q}$
a functor. We say that $F$ is an \emph{equivalence of operads
}if $F$ is fully faithful (which means that it is bijective
on each hom-set) and essentially surjective (which means that $j^{*}(F)$
is an essentially surjective functor of categories). \end{defn}
\begin{rem}
We make a few remarks to emphasize differences and similarities between
the categories $Ope$ and $Cat$:\end{rem}
\begin{itemize}
\item $Ope$ is small complete and small cocomplete. 
\item There is a unique initial operad which is, of course, equal to $j_{!}(\emptyset)$. 
\item For the operad $\star$ above and a terminal category $*$ there holds that
$\star\cong j_{!}(*)$ and $*\cong j^{*}(\star)$. 
\item $\star$ is not terminal but is exponentiable in the category of reduced symmetric operads.
\item The terminal object in $Ope$ is the operad $Comm=j_{*}(*)$ consisting
of one object and one $n$-ary operation for every $n\in\mathbb{N}$. 
\item The subobjects of the terminal operad $Comm$ are all of the following
form. An operad with one object and for every $n\ge0$ at most
one arrow of arity $n$ such that if an arrow of arity $m$ and an
arrow of arity $k$ exist then there is also an arrow of arity $m+k-1$. 
\end{itemize}
A typical example of category is obtained by fixing some mathematical
object and considering the totality of those objects and their naturally
occurring morphisms. In many cases these objects also have a notion
of 'morphism of several variables' in which case the totality of objects
and their multivariable arrows will actually form an operad. One case in which this
is guaranteed is the following.
\begin{lem}
\label{lem:mon cat as ope}Let $(\mathcal{E},\otimes,I)$ be a symmetric
monoidal category and consider for every $x_{0},\cdots,x_{n}\in\mathcal{E}_{0}$
the set $\hat{\mathcal{E}}(x_{1},\cdots,x_{n};x_{0})=\mathcal{E}(x_{1}\otimes\cdots\otimes x_{n},x_{0})$.
With the obvious definitions of composition and identities this construction
defines a symmetric operad $\hat{\mathcal{E}}$ with $(\hat{\mathcal{E}})_{0}=\mathcal{E}_{0}$.\end{lem}
\begin{proof}
The associativity of the composition in $\hat{\mathcal{E}}$ is a
result of the coherence in $\mathcal{E}$.\end{proof}
\begin{rem}
Certainly not every symmetric operad is obtained in that way from
a symmetric monoidal category (e.g., any of the proper subobjects
of $Comm$ or any operad of the form $j_{!}(\mathcal{C})$ for a category
$\mathcal{C}$). It is possible to internally characterize those symmetric
operads that do arise in that way from symmetric monoidal categories, as is explained in detail in \cite{Hermida rep mult cat} and indicated
in \cite{Leinster higher}. 
\end{rem}
Another type of category that arises naturally is one that encodes
some properties of arrows abstractly. For example, the free-living
isomorphism $0\leftrightarrows 1$ is a category with two distinct
objects and, except for the two identities, two other arrows between
the objects, each of which is the inverse of the other. A functor
from the free-living isomorphism to any category $\mathcal{C}$ corresponds
exactly to a choice of an isomorphism in $\mathcal{C}$ and can be
seen as the abstract free-living isomorphism becoming concrete in
the category $\mathcal{C}$. A similar phenomenon is true in operads, where one readily sees the much greater expressive
power of operads compared to categories. Consider for example the
terminal operad $Comm$, for which it is straightforward to prove that
any functor of operads $Comm\to\hat{\mathcal{E}}$ is the same as
a commutative monoid in $\mathcal{E}$, for any symmetric monoidal
category $\mathcal{E}$. There is no category $\mathcal{C}$ with
the property that functors $\mathcal{C}\to\mathcal{E}$ correspond
to commutative monoids in $\mathcal{E}$. 
\begin{rem}
\label{rem:To-distinguish-between}To distinguish between symmetric
operads such as $Comm$ thought of as encoding properties of arrows
and symmetric operads such as $\hat{\mathcal{E}}$ thought of as environments
where operads $\mathcal{P}$ are interpreted concretely we will use
letters near $\mathcal{P}$ for abstract symmetric operads and letters
near $\mathcal{E}$ for symmetric operads as environments (whether
they come from a symmetric monoidal category or not). We will also
call symmetric operads $\mathcal{E}$ \emph{environment operads}.
The distinction is purely syntactic. 
\end{rem}
The utility of operads is in their ability to codify quite a wide
range of algebraic structures in the way described above. The usual
terminology one uses is that of an \emph{algebra }of an operad. The
following definition of algebra is more general than the usual one
(e.g., \cite{MSS book,May GILS}). 
\begin{defn}
Let $\mathcal{P}$ and $\mathcal{E}$ be symmetric operads and consider
a functor $F:\mathcal{P}\to\mathcal{E}$. If $F_{0}:\mathcal{P}_{0}\to\mathcal{E}_{0}$
is the object part of the functor $F$ we say that $F$ is a $\mathcal{P}$-\emph{algebra}
structure on the collection of objects $\{F_{0}(P)\}_{P\in\mathcal{P}_{0}}$
in the environment operad $\mathcal{E}$.
\end{defn}
Many basic properties of $\mathcal{P}$-algebras are captured efficiently
by the introduction of a closed monoidal structure on $Ope$. The
appropriate tensor product of symmetric operads is the Boardman-Vogt
tensor product which was first introduced in \cite{BV book} for (certain
structures that are essentially equivalent to) symmetric operads enriched
in topological spaces. The construction is general enough that it
can be performed for operads enriched in other monoidal categories
and certainly also in the non-enriched case, which is the version
we give now. 
\begin{defn}
Let $\mathcal{P}$ and $\mathcal{Q}$ be two symmetric operads. Their
\emph{Boardman-Vogt tensor product} is the symmetric operad $\mathcal{P}\otimes\mathcal{Q}$
with $(\mathcal{P}\otimes\mathcal{Q})_{0}=\mathcal{P}_{0}\times\mathcal{Q}_{0}$
given in terms of generators and relations as follows. For each $Q\in\mathcal{Q}_{0}$
and each operation $\psi\in\mathcal{P}(P_{1},\cdots,P_{n};P)$ there
is a generator $\psi\otimes Q$ with domain $(P_{1},Q),\cdots,(P_{n},Q)$
and codomain $(P,Q)$. For each $P\in\mathcal{P}_{0}$ and an operation
$\varphi\in\mathcal{Q}(Q_{1},\cdots,Q_{m};Q)$ there is a generator
$P\otimes\varphi$ with domain $(P,Q_{1}),\cdots,(P,Q{}_{m})$ and
codomain $(P,Q)$. There are five types of relations among the arrows ($\sigma$ and $\tau$ below are permutations whose roles are explained below):

1) $(\psi\otimes Q)\circ((\psi_{1}\otimes Q),\cdots,(\psi_{n}\otimes Q))=(\psi\circ(\psi_{1},\cdots,\psi_{n}))\otimes Q$

2) $\sigma^{*}(\psi\otimes Q)=(\sigma^{*}\psi)\otimes Q$

3) $(P\otimes\varphi)\circ((P\otimes\varphi_{1}),\cdots,(P\otimes\varphi_{m}))=P\otimes(\varphi\circ(\varphi_{1},\cdots,\varphi_{m}))$

4) $\sigma^{*}(P\otimes\varphi)=P\otimes(\sigma^{*}\varphi)$

5) $(\psi\otimes Q)\circ((P_{1}\otimes\varphi),\cdots,(P_{n}\otimes\varphi))=\tau^{*}((P\otimes\varphi)\circ((\psi,Q_{1}),\cdots,(\psi,Q{}_{m})))$
\end{defn}
By the relations above we mean every possible choice of arrows $\psi,\varphi,\psi_{i},\varphi_{j}$
for which the compositions are defined. The relations of type 1 and
2 ensure that for any $Q\in\mathcal{P}_{0}$, the map $P\mapsto(P,Q)$
naturally extends to a functor $\mathcal{P}\rightarrow\mathcal{P}\otimes\mathcal{Q}$.
Similarly, the relations of type 3 and 4 guarantee that for each $P\in\mathcal{P}_{0}$,
the map $Q\mapsto(P,Q)$ naturally extends to a functor $\mathcal{Q}\rightarrow\mathcal{P}\otimes\mathcal{Q}$.
The relation of type 5 can be visualized as follows. The left hand
side can be drawn as\[
\xymatrix{*{}\ar@{-}[dr]_{(P_{1},Q_{1})} &  & *{}\ar@{-}[dl]^{(P_{1},Q_{m})} &  & *{}\ar@{-}[dr]_{(P_{n},Q_{1})} &  & *{}\ar@{-}[dl]^{(P_{n},Q_{m})}\\
\ar@{}[r]|{P_{1}\otimes\varphi} & *{\bullet}\ar@{-}[drr]_{(P_{1},Q)} &  &  & \ar@{}[r]|{P_{n}\otimes\varphi} & *{\bullet}\ar@{-}[dll]^{(P_{n},Q)}\\
 &  & \ar@{}[r]|{\psi\otimes Q} & *{\bullet}\ar@{-}[d]^{(P,Q)}\\
 &  &  & *{}}
\]
while the right hand side can be drawn as 
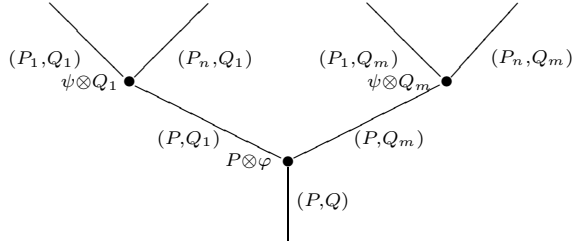
\[
\xymatrix{*{}\ar@{-}[dr]_{(P_{1},Q_{1})} &  & *{}\ar@{-}[dl]^{(P_{n},Q_{1})} &  & *{}\ar@{-}[dr]_{(P_{1},Q_{m})} &  & *{}\ar@{-}[dl]^{(P_{n},Q_{m})}\\
\ar@{}[r]|{\psi\otimes Q_{1}} & *{\bullet}\ar@{-}[drr]_{(P,Q_{1})} &  &  & \ar@{}[r]|{\psi\otimes Q_{m}\,\,\,\,} & *{\bullet}\ar@{-}[dll]^{(P,Q_{m})}\\
 &  & \ar@{}[r]|{P\otimes\varphi} & *{\bullet}\ar@{-}[d]^{(P,Q)}\\
 &  &  & *{}}
\]
As given, the operations cannot be equated since their domains do
not agree. There is however an evident permutation $\tau$ that equates
the domains and it is that permutation $\tau$ that is used in the
equation of type $5$ above. 
\begin{thm}
The category $(Ope,\otimes,\star)$ is a symmetric closed monoidal
category. \end{thm}
\begin{proof}
The internal hom operad $[\mathcal{P},\mathcal{Q}]$ has as objects
all morphisms of operads $F:\mathcal{P}\to\mathcal{Q}$ and the arrows
with domain $F_{1},\cdots,F_{n}$ and codomain $F_{0}$ are analogues
of natural transformations as follows. A \emph{natural transformation}
$\alpha$ from $(F_{1},\cdots,F_{n})$ to $F_{0}$ is a family $\{\alpha_{P}\}_{P\in\mathcal{P}_{0}}$,
with $\alpha_{P}\in\mathcal{Q}(F_{1}(P),\cdots,F_{n}(P);F_{0}(P))$,
satisfying the following property. Given any operation $\psi\in\mathcal{P}(P_{1},\cdots,P_{m};P)$
consider the following diagrams in $\mathcal{Q}$: 
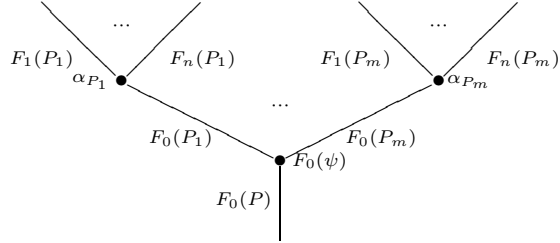
\[
\xymatrix{*{}\ar@{-}[rd]_{F_{1}(P_{1})} & \ar@{}[d]|<<<{\cdots} & *{}\ar@{-}[dl]^{F_{n}(P_{1})} &  & *{}\ar@{-}[dr]_{F_{1}(P_{m})} & \ar@{}[d]|<<<{\cdots} & *{}\ar@{-}[dl]^{F_{n}(P_{m})}\\
\ar@{}[r]|{\,\,\,\,\,\alpha_{P_{1}}} & *{\bullet}\ar@{-}[drr]_{F_{0}(P_{1})} &  & \ar@{}[d]|<<<{\cdots} &  & *{\bullet}\ar@{-}[dll]^{F_{0}(P_{m})} & \ar@{}[l]|{\alpha_{P_{m}}\,\,\,\,\,}\\
 &  &  & *{\bullet}\ar@{-}[d]_{F_{0}(P)}\ar@{}[r]|{F_{0}(\psi)} & *{}\\
 &  &  & *{}}
\]
 and 
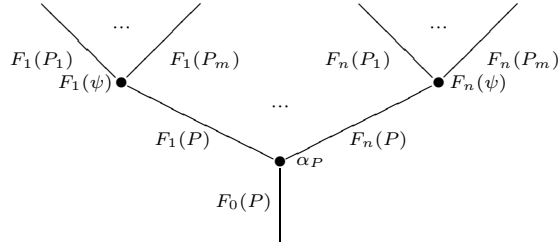
\[
\xymatrix{*{}\ar@{-}[rd]_{F_{1}(P_{1})} & \ar@{}[d]|<<<{\cdots} & *{}\ar@{-}[dl]^{F_{1}(P_{m})} &  & *{}\ar@{-}[dr]_{F_{n}(P_{1})} & \ar@{}[d]|<<<{\cdots} & *{}\ar@{-}[dl]^{F_{n}(P_{m})}\\
\ar@{}[r]|{\,\, F_{1}(\psi)} & *{\bullet}\ar@{-}[drr]_{F_{1}(P)} &  & \ar@{}[d]|<<<{\cdots} &  & *{\bullet}\ar@{-}[dll]^{F_{n}(P)} & \ar@{}[l]|{F_{n}(\psi)}\\
 &  &  & *{\bullet}\ar@{-}[d]_{F_{0}(P)}\ar@{}[r]|{\alpha_{P}\,\,\,\,\,} & *{}\\
 &  &  & *{}}
\]
and let $\varphi_{1}$ and $\varphi_{2}$ be their respective compositions.
Then $\varphi_{2}=\sigma^{*}(\varphi_{1})$, where $\sigma$ is the
evident permutation equating the domain of $\varphi_{1}$ with that
of $\varphi_{2}$. The interested reader is referred to \cite{thesis}
for more details on horizontal and vertical compositions of natural
transformations leading to the construction of the strict $2$-category
of small operads in which the strict $2$-category of small categories
embeds. 
\end{proof}
We now return to our general notion of $\mathcal{P}$-algebras in
$\mathcal{E}$ and notice the very simple result:
\begin{lem}
Let $\mathcal{P}$ and $\mathcal{E}$ by symmetric operads. The internal
hom $[\mathcal{P},\mathcal{E}]$ is rightfully to be called the operad
of $\mathcal{P}$-algebras in $\mathcal{E}$ in the sense that the
objects of $[\mathcal{P},\mathcal{E}]$ are the $\mathcal{P}$-algebras
in $\mathcal{E}$, the unary arrows are the morphisms of such algebras,
and the $n$-ary arrows are 'multivariable' morphisms of algebras
(with $0$-ary morphisms thought of as constants). 
\end{lem}
It is trivial to verify for example that $[Comm,Set]$ is isomorphic
to the operad obtained from the symmetric monoidal category $CommMon(Set)$
of commutative monoids in $Set$ by means of the construction given
in Lemma \ref{lem:mon cat as ope}. Here $Set$ can be replaced by
any symmetric monoidal category. This motivates the following definition. 
\begin{defn}
Let $\mathcal{E}$ be a symmetric operad and $S$ some notion of an
algebraic structure on objects of $\mathcal{E}$ together with a notion
of (perhaps multivariable) morphisms between such structures. We call
a symmetric operad $\mathcal{P}$ a \emph{classifying }operad for
$S$ (in $\mathcal{E}$) if the operad $[\mathcal{P},\mathcal{E}]$
satisfies that $[\mathcal{P},\mathcal{E}]_{0}$ is precisely the set
of $S$-structures in $\mathcal{E}$ and the arrows in $[\mathcal{P},\mathcal{E}]$
correspond precisely to the notion of morphisms between such structures. \end{defn}
\begin{example}
\label{exa:classifying categories over A}The symmetric operad $Comm$
is a classifying operad for commutative monoids in a symmetric operad
$\mathcal{E}$. There is a symmetric operad $As$ that classifies
monoids in $\mathcal{E}$ (i.e., an object with an associative binary
operation with a unit) which the reader is invited to find. A \emph{magma}
is a set together with a binary operation, not necessarily associative, 
and there is a symmetric operad that classifies magmas. There is also
a symmetric operad that classifies non-unital monoids as well as one
that classifies non-unital commutative monoids. It is a rather unfortunate
fact that there is no symmetric operad that classifies all small categories.
However, given a fixed set $A$ consider the category $Cat_{A}$ of
\emph{categories over $A$, }in which the objects are categories having
$A$ as set of objects and where the arrows are functors between such
categories whose object part is the identity. Then there is a symmetric
operad $C_{A}$ that classifies categories over $A$. Similarly, with
the obvious definition, there is a symmetric operad $O_{A}$ that
classifies symmetric operads over $A$. \end{example}
\begin{rem}
In general, there can be two non-equivalent operads $\mathcal{P}$
and $\mathcal{Q}$ that classify the same algebraic structure. We
will not get into the question of detecting when two symmetric operads
have equivalent operads of algebras. 
\end{rem}
A well-known phenomenon in category theory is the interchangeability
of repeated structures. Thus, for example, a category object in $Grp$
is the same as a group object in $Cat$. With the formalism of symmetric
operads we have thus far we can easily prove a whole class of such
cases (but in fact not the case just mentioned, since group objects
cannot be classified by symmetric operads).
\begin{lem}
Let $\mathcal{P}_{1}$ and $\mathcal{P}_{2}$ be two symmetric operads
and let $\mathcal{E}$ be an environment operad. Then $\mathcal{P}_{1}$-algebras
in $\mathcal{P}_{2}$-algebras in $\mathcal{E}$ are the same as $\mathcal{P}_{2}$-algebras
in $\mathcal{P}_{1}$-algebras in $\mathcal{E}$. \end{lem}
\begin{proof}
The precise formulation of the lemma is that there is an isomorphism
of symmetric operads $[\mathcal{P}_{1},[\mathcal{P}_{2},\mathcal{E}]]\cong[\mathcal{P}_{2},[\mathcal{P}_{1},\mathcal{E}]]$.
The proof is trivial from the symmetry of the Boardman-Vogt tensor
product.
\end{proof}
Consider the operad $As$ that classifies monoids: it has just one
object and its arrows of arity $n$ is the set $\Sigma_{n}$ of permutations
on $n$-symbols. It is not hard to show that $As\otimes As\cong Comm$
which essentially is Eckman-Hilton duality proving that associative
monoids in associative monoids are commutative monoids, except that
it is done at the level of classifying operads rather than algebras. 

We conclude our review of the basics of operad theory by noting that
in the same way that categories can be enriched in a symmetric monoidal category
$\mathcal{E}$ (see \cite{Kelly Enr. Cat. }) so can operads be so
enriched. With the evident definitions one then obtains the category
$Ope(\mathcal{E})$ of all small operads enriched in $\mathcal{E}$. 
\begin{rem}
In the presence of coproducts in $\mathcal{E}$ any non-enriched symmetric
operad $\mathcal{P}$ gives rise to an operad $Dis(\mathcal{P})$
enriched in $\mathcal{E}$ in which each hom-object is a coproduct,
indexed by the corresponding hom-set in $\mathcal{P}$, of the unit
$I$ of $\mathcal{E}$. We will usually refer to $Dis(\mathcal{P})$
as the corresponding discrete operad in $\mathcal{E}$ and call it
again $\mathcal{P}$. 
\end{rem}
Our main interest in symmetric operads is in their use in the theory
of homotopy invariant algebraic structures where enrichment plays
a vital role. However, before we embark on the subtleties of homotopy
invariance we briefly treat the isomorphism invariance property for
non-enriched symmetric operads.

\subsection{The isomorphism invariance property }

It is a triviality that an algebraic structure can be transferred,
uniquely, along an isomorphism. To be more precise and to formulate
this in the language of operads, let $\mathcal{P}$ and $\mathcal{E}$
be symmetric operads and $F:\mathcal{P}\to\mathcal{E}$ an algebra
structure on $\{F_{0}(P)\}_{P\in\mathcal{P}_{0}}$. Assume that we are
given a family $\{f_{P}:F_{0}(P)\to G_{0}(P)\}_{P\in\mathcal{P}_{0}}$
of isomorphisms in $\mathcal{E}$. Then there exists a unique $\mathcal{P}$-algebra
structure $G:\mathcal{P}\to\mathcal{E}$ on $\{G_{0}(P)\}_{P\in\mathcal{P}_{0}}$
for which the family $\{f_{P}\}_{P\in\mathcal{P}_{0}}$ forms a natural
isomorphism from $F$ to $G$ and thus an isomorphism between the
algebras. We call this the \emph{isomorphism invariance property }of
algebras. 

We can reformulate this property diagrammatically as follows. Let
$0$ be a one-object symmetric operad with the identity arrow only,
and $0\to(0\leftrightarrows1)$ the inclusion $0\mapsto0$ into the
free-living isomorphism. Then a choice of functor $F:\mathcal{P}\to\mathcal{E}$
is the same as a functor $0\to[\mathcal{P},\mathcal{E}]$ while a
functor $(0\leftrightarrows1)\to[\mathcal{P},\mathcal{E}]$ can be
identified with two functors $\xymatrix{\mathcal{P}\ar@<2pt>[r]\ar@<-2pt>[r] & \mathcal{E}}
$ and a natural isomorphism between them. The set of objects $\mathcal{P}_{0}$
seen as a category with only identity arrows can be seen as a symmetric
operad. One then has the evident inclusion functor $\mathcal{P}_{0}\to\mathcal{P}$
which induces a functor $[\mathcal{P},\mathcal{E}]\to[\mathcal{P}_{0},\mathcal{E}]$.
The isomorphism invariance property for $\mathcal{P}$-algebras in
$\mathcal{E}$ is then the statement that in the following diagram:\[
\xymatrix{0\ar[d]\ar[rrr]^{\forall F} &  &  & [\mathcal{P},\mathcal{E}]\ar[d]\\
0\leftrightarrows1\ar[rrr]_{\forall\{F_{0}(P)\to G_{0}(P)\}_{P\in\mathcal{P}_{0}}}\ar@{..>}[rrru]^{\exists\alpha} &  &  & [\mathcal{P}_{0},\mathcal{E}]}
\]
the diagonal filler exists (and is unique) for any functor $F:\mathcal{P}\to \mathcal{E}$ and any family of isomorphisms
$\{F_{0}(P)\to G_{0}(P)\}_{P\in\mathcal{P}_{0}}$. 

In the formalism of Quillen model structures there is a conceptual
way to see why a lift in the diagram above exists. To present it we
recall that a functor $F:\mathcal{C}\to\mathcal{D}$ of categories
is an \emph{isofibration} if it has the right lifting property with
respect to the inclusion $0\to(0\leftrightarrows1)$. Similarly, a
functor $F:\mathcal{P}\to\mathcal{Q}$ of symmetric operads is an
\emph{isofibration} of symmetric operads if it has the right lifting
property with respect to the same inclusion $0\to(0\leftrightarrows1)$
with each category seen as an operad. Equivalently, $F:\mathcal{P}\to\mathcal{Q}$
is an isofibration (of operads) if, and only if, $j^{*}(F)$ is an
isofibration of categories. We now recall the operadic Quillen model
structure on symmetric operads. 
\begin{thm}
The category $Ope$ of symmetric operads with the Boardman-Vogt tensor
product admits a cofibrantly generated closed monoidal model structure
in which the weak equivalences are the operadic equivalences, the
cofibrations are those functors $F:\mathcal{P}\to\mathcal{Q}$ such
that the object part of $F$ is injective, and the fibrations are
the isofibrations. All operads are fibrant and cofibrant. The Quillen
model structure induced on $Cat\cong Ope/\star$ is the categorical
one (also known as the 'folk' or 'natural' model structure).\end{thm}
\begin{proof}
A direct verification of the axioms of a model category is not difficult
and not too tedious. Further details can be found in \cite{thesis}. 
\end{proof}
Now, in the diagram above the left vertical arrow is a trivial cofibration
and the right vertical arrow is, by the monoidal model structure axiom,
a fibration and hence the lift exists. We summarize the above discussion:
\begin{fact}
\label{fac:iso invar}The notion of algebras of operads is internalized
to the category $Ope$ by it being closed monoidal with respect to
the Boardman-Vogt tensor product. The isomorphism invariance property
of algebras is captured by the operadic Quillen model structure and
its compatibility with the Boardman-Vogt tensor product. 
\end{fact}

\subsection{The homotopy invariance property}

In the presence of homotopy in $\mathcal{E}$ one can ask if a stronger
property than the isomorphism invariance property holds. Namely, if
one merely asks for the arrows $F_{0}(P)\to G_{0}(P)$ to be weak equivalences
instead of isomorphisms is it still possible to transfer the algebra
structure? A simple example is when one considers a topological monoid
$X$ and a topological space $Y$ together with continuous mappings
$f:X\to Y$ and $g:Y\to X$ such that $f\circ g$ and $g\circ f$
are homotopic to the respective identities. It is evident that if
$f$ and $g$ are not actual inverses of each other then the monoid
structure on $X$ will not, in general, induce a monoid structure
on $Y$. The question as to what kind of structure is induced goes
back to Stasheff's study of $H$-spaces and his famous associahedra
that are used to describe the kind of structure that arises \cite{H spaces}.
The more general problem for algebraic structures on topological spaces
can be addressed by using enriched symmetric operads as is done by
Boardman and Vogt in \cite{BV book}. Their techniques and results
were generalized by Berger and Moerdijk in a series of three papers
\cite{ax hom th ope,BV res ope,res colour ope} and below we present
an expository account of the constructions we will need. 

First we give a slightly vague definition of the homotopy invariance
property. In the context of dendroidal sets below we will give a precise
definition that is completely analogous to the definition of the isomorphism
invariance property. 
\begin{defn}
Let $\mathcal{E}$ be a symmetric monoidal model category and $\mathcal{Q}$ a symmetric
operad enriched in $\mathcal{E}$. We say that $\mathcal{Q}$-algebras
have the \emph{homotopy invariance property }if given an algebra $F:\mathcal{Q}\to\hat{\mathcal{E}}$
on $\{F_{0}(Q)\}_{Q\in\mathcal{Q}_{0}}$ and a family $\{f_{Q}:F_{0}(Q)\to G_{0}(Q)\}{}_{Q\in\mathcal{Q}_{0}}$
of weak equivalences in $\mathcal{E}$ (with perhaps some extra conditions)
there exists an essentially unique $\mathcal{Q}$-algebra structure
$G:\mathcal{Q}\to\hat{\mathcal{E}}$ on $\{G_{0}(Q)\}_{Q\in\mathcal{Q}_{0}}$. 
\end{defn}
It is evident that an arbitrary symmetric operad $\mathcal{P}$ need
not have the homotopy invariance property and the problem of sensibly
replacing $\mathcal{P}$ by another operad $\mathcal{Q}$ that does
have this property is referred to as the problem of finding the up-to-homotopy
version of the algebraic structure classified by $\mathcal{P}$. To
make this notion precise we recall that in \cite{ax hom th ope,BV res ope,res colour ope}
Berger and Moerdijk establish the following result.
\begin{thm}
Let $\mathcal{E}$ be a cofibrantly generated symmetric monoidal model category. Under
mild conditions the category $Ope(\mathcal{E})_{A}$ of symmetric
operads enriched in $\mathcal{E}$ with fixed set of objects equal
to $A$ and whose functors are the identity on all objects admits
a Quillen model structure in which the weak equivalences are hom-wise
weak equivalences and the fibrations are hom-wise fibrations. 
\end{thm}
We refer to this model structure as the Berger-Moerdijk model structure
on $Ope(\mathcal{E})_{A}$. 
\begin{rem}
The Berger-Moerdijk model structure on symmetric operads over a singleton
$A=\{*\}$ (given in \cite{ax hom th ope}) settles one of the open
problems listed by Hovey in \cite{Hovey Model Cat}.
\end{rem}
Among the consequences of the model structure Berger and Moerdijk
prove the following. 
\begin{thm}
If $\mathcal{Q}$ is cofibrant in the Berger-Moerdijk model structure
on $Ope(\mathcal{E})_{A}$ then $\mathcal{Q}$-algebras in $\mathcal{E}$
have, under mild conditions, the homotopy invariance property. \end{thm}
\begin{proof}
See Theorem 3.5 in \cite{ax hom th ope} for more details. 
\end{proof}
Thus, the problem of finding the up-to-homotopy version of the algebraic
structure classified by a symmetric operad $\mathcal{P}$ enriched
in $\mathcal{E}$ reduces to finding a cofibrant replacement $\mathcal{Q}$
of $\mathcal{P}$ in the Berger-Moerdijk model structure on $Ope(\mathcal{E})_{\mathcal{P}_{0}}$.
Of course, a cofibrant replacement always exists just by the presence
of the Quillen model structure. However, in order to actually compute
with it one needs an efficient construction of it, and this is the aim
of the $W$-construction.

\subsubsection{The original Boardman-Vogt W-construction for topological operads}

The $W$-construction is a functor $W:Ope(Top)\rightarrow Ope(Top)$
equipped with a natural transformation (an augmentation) $W\rightarrow id$.
A detailed account (albeit in a slightly different language than that
of operads) can be found in \cite{BV book} where it first appeared.
We give here an expository presentation aiming at explaining the ideas
important to us. 

For simplicity let us describe the planar version of the $W$-construction,
that is, we describe a functor taking a planar operad enriched in $Top$
to another such planar operad. We now fix a topological planar operad
$\mathcal{P}$ and describe the operad $W\mathcal{P}$. The objects
of $W\mathcal{P}$ are the same as those of $\mathcal{P}$. To describe
the arrow spaces we consider standard planar trees (a tree is planar
when it comes with an orientation of the edges at each vertex and
\emph{standard }means that a choice was made of a single planar tree
of each isomorphism class of planar isomorphisms of planar trees)
whose edges are labelled by objects of $\mathcal{P}$ and whose vertices
are labelled by arrows of $\mathcal{P}$ according to the rule that
the objects labelling the input edges of a vertex are equal (in their
natural order) to the input of the operation labelling that vertex.
Similarly the object labelling the output of the vertex is the output
object of the operation at the vertex. Moreover, each inner edge in
such a tree is given a length $0\le t\le1$. For objects $P_{0},\cdots,P_{n}\in (W\mathcal{P})_{0}$
let $A(P_{1},\cdots,P_{n};P_{0})$ be the topological space whose
underlying set is the set of all such planar labelled trees $\bar{T}$
for which the leaves of $\bar{T}$ are labelled by $P_{1},\cdots,P_{n}$
(in that order) and the root of $\bar{T}$ is labelled by $P_{0}$.
The topology on $A(P_{1},\cdots,P_{n};P_{0})$ is the evident one
induced by the topology of the arrow spaces in $\mathcal{P}$ and
the standard topology on the unit interval $[0,1]$. 

The space $W\mathcal{P}(P_{1},\cdots,P_{n};P_{0})$ is the quotient
of $A(P_{1},\cdots,P_{n};P_{0})$ obtained by the following identifications.
If $\bar{T}\in A(P_{1},\cdots,P_{n};P_{0})$ has an inner edge $e$
whose length is $0$ then we identify it with the tree $\bar{T}/e$
obtained from $\bar{T}$ by contracting the edge $e$ and labelling
the newly formed vertex by the corresponding $\circ_{i}$-composition
of the operations labelling the vertices at the two sides of $e$
(the other labels are as in $\bar{T}$). Thus pictorially we have
that locally in the tree a configuration 
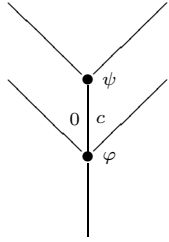
\[
\xymatrix{*{\,}\ar@{-}[dr] &  & *{\,}\ar@{-}[dl]\\
*{\,}\ar@{-}[dr] & *{\bullet}\ar@{-}[d]_{0}^{c} & *{\,}\ar@{-}[dl]\ar@{}[l]|{\psi\,\,\,\,\,\,\,\,\,\,}\\
 & *{\bullet}\ar@{-}[d] & \ar@{}[l]|{\varphi\,\,\,\,\,\,\,\,\,\,}\\
 & *{\,}}
\]
is identified with the configuration 
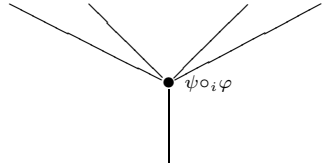
\[
\xymatrix{*{\,}\ar@{-}[drr] & *{\,}\ar@{-}[dr] &  & *{\,}\ar@{-}[dl] & *{\,}\ar@{-}[dll]\\
 &  & *{\bullet}\ar@{-}[d] & \ar@{}[l]|{\psi\circ_{i}\varphi}\\
 &  & *{\,}}
\]
Another identification is in the case of a tree $\bar{S}$ with a
unary vertex $v$ labelled by an identity. We identify such a tree
with the tree $\bar{R}$ obtained by removing the vertex $v$ and
identifying its input edge with its output edge. The length
assigned to the new edge is determined as follows. If it is an outer
edge then it has no length. If it is an inner edge then it is assigned
the maximum of the lengths of $s$ and $t$ (where if either $s$
or $t$ does not have a length, i.e., it is an outer edge, then its
length is considered to be $0$). The labelling is as in $\bar{S}$
(notice that the label of the newly formed edge is unique since $v$
was labelled by an identity which means that its input and output
were labelled by the same object). Pictorially, this identification
identifies the labelled tree 
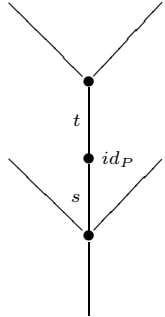
\[
\xymatrix{*{\,}\ar@{-}[dr] &  & *{\,}\ar@{-}[dl]\\
 & *{\bullet}\ar@{-}[d]_{t}\\
*{\,}\ar@{-}[dr] & *{\bullet}\ar@{-}[d]_{s} & *{\,}\ar@{-}[dl]\ar@{}[l]|{id_{P}\,\,\,\,}\\
 & *{\bullet}\ar@{-}[d]\\
 & *{\,}\\
\\}
\]
 with the tree 
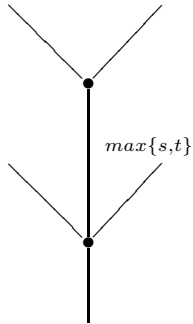
\[
\xymatrix{*{\,}\ar@{-}[dr] &  & *{\,}\ar@{-}[dl]\\
 & *{\bullet}\ar@{-}[dd]\\
*{\,}\ar@{-}[dr] &  & *{\,}\ar@{-}[dl]\ar@{}[l]_{\,\,\,\,\,\,\,\,\,\,\,\,\, max\{s,t\}}\\
 & *{\bullet}\ar@{-}[d]\\
 & *{\,}}
\]
The composition in $W\mathcal{P}$ is given by grafting such labelled
trees, giving the newly formed inner edge length $1$. The augmentation
$W\mathcal{P}\to\mathcal{P}$ is the identity on objects and sends
an arrow represented by such a labeled tree to the operation obtained
by contracting all lengths of internal edges to $0$ and composing
in $\mathcal{P}$. We leave the necessary adaptations needed for obtaining
the symmetric version of the $W$-construction to the reader. 
\begin{example}
Let $\mathcal{P}$ be the planar operad with a single object and a
single $n$-ary operation in each arity $n\ge 1$ and no arrows of arity $0$. We consider $\mathcal{P}$
to be a discrete operad in $Top$. It is easily seen that a functor
$\mathcal{P}\rightarrow Top$ corresponds to a non-unital topological monoid (we treat this case for simplicity).
Let us now calculate the first few arrow spaces in $W\mathcal{P}$.
Firstly, $W\mathcal{P}$ too has just one object. We thus use the
notation of classical operads, namely $W\mathcal{P}(n)$ for the space
of operations of arity $n$. Clearly $W\mathcal{P}(0)$ is just 
the empty space. The space $W\mathcal{P}(1)$ consists of labelled
trees with one input. Since in such a tree the only possible label
at a vertex is the identify, the identification regarding identities
implies that $W\mathcal{P}(1)$ is again just a one-point space. In
general, since every unary vertex in a labelled tree in $W\mathcal{P}(n)$
can only be labelled by the identity, and those are then identified
with trees not containing unary vertices, it suffices to only consider
reduced trees, namely trees with no unary vertices. To calculate $W\mathcal{P}(2)$
we need to consider all reduced trees with two inputs, but there is
just one such tree, the $2$-corolla, and it has no inner edges, thus
$W\mathcal{P}(2)$ is also a one-point space. Things become more interesting
when we calculate $W\mathcal{P}(3)$. We need to consider reduced
trees with three inputs. There are three such trees, namely\[
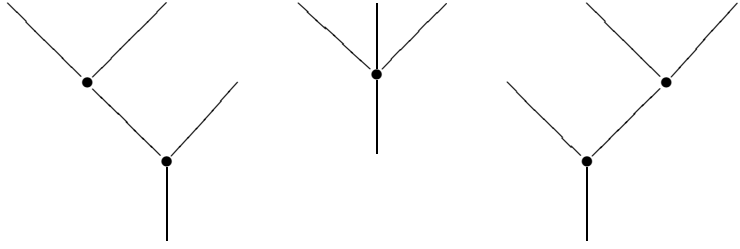
\begin{array}{ccc}
\xymatrix{*{}\ar@{-}[dr] &  & *{}\ar@{-}[dl]\\
 & *{\bullet}\ar@{-}[dr] &  & *{}\ar@{-}[dl]\\
 &  & *{\bullet}\ar@{-}[d]\\
 &  & *{}}
 & \quad\xymatrix{*{}\ar@{-}[dr] & *{}\ar@{-}[d] & *{}\ar@{-}[dl]\\
 & *{\bullet}\ar@{-}[d]\\
 & *{}}
\quad & \xymatrix{ & *{}\ar@{-}[dr] &  & *{}\ar@{-}[dl]\\
*{}\ar@{-}[dr] &  & *{\bullet}\ar@{-}[dl]\\
 & *{\bullet}\ar@{-}[d]\\
 & *{}}
\end{array}\]
The middle tree contributes a point to the space $W\mathcal{P}(3)$.
Each of the other trees has one inner edge and thus contributes the
interval $[0,1]$ to the space. The only identification to be made
is when the length of one of those inner edges is $0$, in which case
it is identified with the point corresponding to the middle tree.
The space $W\mathcal{P}(3)$ is thus the gluing of two copies of the
interval $[0,1]$ where we identify both ends named $0$ to a single
point. The result is then just a closed interval, $[-1,1]$. However,
it is convenient to keep in mind the trees corresponding to each point
of this interval. Namely, the tree corresponding to the middle point,
$0$, is the middle tree. With a point $0<t\le1$ corresponds the
tree on the right where the length of the inner edge is $t$, and
with a point $-1\le-t<0$ corresponds the tree on the left where its
inner edge is given the length $t$. In this way one can calculate
the entire operad $W\mathcal{P}$. It can then be shown that the spaces
$\{W\mathcal{P}(n)\}_{n=0}^{\infty}$, reproduce, up to homeomorphism, the
Stasheff associahedra. An $A_{\infty}$-space is then an algebra over
$W\mathcal{P}$ and $W\mathcal{P}$ classifies $A_{\infty}$-spaces
and their \emph{strong} morphisms. 
\end{example}

\subsubsection{The Berger-Moerdijk generalization of the W-construction to operads
enriched in a homotopy environment.}

Observe that in the $W$-construction given above one can construct
the space $W\mathcal{P}(P_{1},\cdots,P_{n};P_{0})$ as follows. For
each labelled planar tree $\bar{T}$ as above let $H^{\bar{T}}$ be
$H^{k}$ where $k$ is the number of inner edges in $\bar{T}$ and
$H=[0,1]$, the unit interval. Further, for each vertex $v$ of $\bar{T}$
let $\mathcal{P}(v)=\mathcal{P}(x_{1},\cdots,x_{n};x_{0})$ where
$x_{1},\cdots,x_{n}$ are (in that order) the inputs of $v$ and $x_{0}$
its output. Finally, let $\mathcal{P}(\bar{T})$ be the product of
$\mathcal{P}(v)$ where $v$ ranges over the vertices of $\bar{T}$.
Now, The space $A(P_{1},\cdots,P_{n};P_{0})$ constructed above is
homeomorphic to $\coprod_{\bar{T}}(H^{\bar{T}}\times\mathcal{P}(\bar{T}))$
where $\bar{T}$ varies over all labelled standard planar trees $\bar{T}$
whose leaves are labelled by $P_{1},\cdots,P_{n}$ and whose root
is labelled by $P_{0}$. The identifications that are then made to
construct the space $W\mathcal{P}(P_{1},\cdots,P_{n};P_{0})$ are
completely determined by the combinatorics of the various trees $\bar{T}$.
This observation is the key to generalizing the $W$-construction
to symmetric operads in monoidal model categories $\mathcal{E}$ other than
$Top$ and is carried out in \cite{BV res ope,res colour ope}. What
is needed is a suitable replacement for the unit interval $[0,1]$
used above to assign lengths to the inner edges of the trees. Such
a replacement is the notion of an interval object in a monoidal model
category $\mathcal{E}$ given in \cite{BV res ope}.
\begin{defn}
Let $\mathcal{E}$ be a symmetric monoidal model category $\mathcal{E}$
with unit $I$. An\emph{ interval }object in $\mathcal{E}$ (see Definition
4.1 in \cite{BV res ope}) is a factorization of the codiagonal $I\coprod I\to I$
into a cofibration $I\coprod I\to H$ followed by a weak equivalence
$\epsilon:H\to I$ together with an associative operation $\vee:H\otimes H\to H$
which has a neutral element, an absorbing element, and for which $\epsilon$
is a counit. For convenience, when an interval element is chosen we
will refer to $(\mathcal{E},H)$ as a \emph{homotopy environment. }
\end{defn}
Relevant examples to our presentation are the ordinary unit interval
in $Top$ with the standard model structure (with $x\vee y=max\{x,y\}$)
and the free-living isomorphism $0\leftrightarrows1$ in $Cat$ with
the categorical model structure. 

In such a setting the topological $W$-construction can be mimicked
by gluing together objects $H^{\otimes_{k}}$ instead of cubes $[0,1]^{k}$.
This is done in detail in \cite{BV res ope}, to which the interested
reader is referred. We thus obtain a functor $W_{H}:Ope(\mathcal{E})\rightarrow Ope(\mathcal{E})$
for any homotopy environment $\mathcal{E}$. Usually we will just
write $W$ instead of $W_{H}$, which is quite a harmless convention
since Proposition 6.5 in \cite{BV res ope} guarantees that under
mild conditions a different choice of interval object yields essentially
equivalent $W$-constructions. 
\begin{example}
Consider the category $Cat$ with the categorical model structure.
In this monoidal model category we can choose the category $H$ to
be the free-living isomorphism $0\leftrightarrows1$ as interval object,
with the obvious structure maps. Let us again consider the planar
operad $\mathcal{P}$ classifying non-unital associative monoids, this time as a discrete operad in $Cat$. To calculate
$W\mathcal{P}(n)$ we should again consider labelled standard planar
trees with lengths. The same argument as above implies that we should
only consider reduced trees, and a similar calculation shows that
$W\mathcal{P}(n)$ is a one-point category for $n=1,2$. Now, to
calculate $W\mathcal{P}(3)$ we again consider the three trees as
given above. This time the middle tree contributes the category $H^{0}=I$.
Each of the other trees contributes the category $H$. The identifications
identify the object named $0$ in each copy of $H$ with the unique
object of $I$. The result is a contractible category with three objects.
In general, the category $W\mathcal{P}(n)$ is a contractible category
with $tr(n)$ objects, where $tr(n)$ denotes the number of reduced
standard planar trees with $n$ leaves. The composition in $W\mathcal{P}$
is given by grafting of such trees. The operad $W\mathcal{P}$ classifies
unbiased monoidal categories and strict monoidal functors (an unbiased
monoidal category is a category with an $n$-ary multiplication functor
for each $n\ge0$ together with some coherence conditions. See \cite{Leinster higher}
for more details as well as a discussion about the equivalence of
such categories and ordinary weak monoidal categories). 
\end{example}
The generalized Boardman-Vogt $W$-construction thus provides a computationally
tractable way to classify weak algebras for a wide variety of structures
in a homotopy environment. However, $W\mathcal{P}$ tends to classify
weak $\mathcal{P}$-algebra with their strong morphisms and not with
their weak morphisms. Indeed, for some fixed homotopy environment
$\mathcal{E}$ assume that $Ope(\mathcal{E})$ is closed monoidal
with respect to a Boardman-Vogt type tensor product. If we now consider
for a symmetric operad $\mathcal{P}\in Ope(\mathcal{E})_{0}$ the
internal hom $[W\mathcal{P},\mathcal{E}]$ then the elements of $[W\mathcal{P},\mathcal{E}]_{0}$
are precisely the weak $\mathcal{P}$-algebras in $\mathcal{E}$.
However, a unary arrow in $[W\mathcal{P},\mathcal{E}]$ corresponds to
a map of symmetric operads $(W\mathcal{P})\otimes[1]\to\mathcal{E}$
(where $[1]$ is the operad $0\to1$ considered as a discrete operad
in $\mathcal{E}$). This already shows that the notion one gets is
of strong (because $W$ does not act on $[1]$) morphisms between
weak (because $W$ does act on $\mathcal{P}$) $\mathcal{P}$-algebras.

\subsection{Weak maps between weak algebras}

Luckily, to arrive at the right notion of weak morphisms between weak
algebras no extra work is needed. Following on the observation above
we make the following definition.
\begin{defn}
Let $\mathcal{P}$ be an operad in $Set$ and $\mathcal{E}$ a homotopy
environment. A \emph{weak $\mathcal{P}$-algebra} in $\mathcal{E}$
is a functor of symmetric $\mathcal{E}$-enriched operads $W(\mathcal{P})\to\hat{\mathcal{E}}$.
A\emph{ weak map} between up-to-homotopy $\mathcal{P}$-algebras in
$\mathcal{E}$ is a functor of symmetric $\mathcal{E}$-enriched operads
$W(\mathcal{P}\otimes[1])\to\hat{\mathcal{E}}.$ 
\end{defn}
An obvious question now is whether the collection of all weak $\mathcal{P}$-algebras
and their weak maps forms a category. The answer is that they usually
do not. A simple example is provided by $A_{\infty}$-spaces where
it is known that weak $A_{\infty}$-maps do not compose associatively.
The theory so far already suggests a solution to that problem. We
denote by $[n]$ the operad $0\to1\to\cdots\to n$ seen as a discrete
operad in $\mathcal{E}$. For a symmetric operad $\mathcal{P}$ in
$Set$ consider the symmetric operad $\mathcal{P}\otimes[n]$. An
algebra for such a symmetric operad is easily seen to be a sequence
$X_{0},\cdots,X_{n}$ of $\mathcal{P}$-algebras together with weak
$\mathcal{P}$-algebra maps:\[
X_{0}\rightarrow X_{1}\rightarrow\cdots\rightarrow X_{n}\]
and all their possible compositions.
\begin{prop}
Let $\mathcal{P}$ be a symmetric operad in $Set$ and $\mathcal{E}$
a homotopy environment. For each $n\ge0$ let $X_{n}$ be the set
of maps \[
W(\mathcal{P}\otimes[n])\to\hat{\mathcal{E}}\]
of symmetric operads enriched in $\mathcal{E}$. Then the collection
$X=\{X_{n}\}_{n=0}^{\infty}$ can be canonically made into a simplicial
set. \end{prop}
\begin{proof}
The proof follows easily by noting that the sequence $\{\mathcal{P}\otimes[n]\}_{n=0}^{\infty}$
is a cosimplicial object in $Ope$. \end{proof}
\begin{defn}
\label{def:SimpSetOfWeakAlg}We refer to the simplicial set constructed
above as the \emph{simplicial set of weak $\mathcal{P}$-algebras}
in $\mathcal{E}$ and denote it by $wAlg[\mathcal{P},\mathcal{E}]$. 
\end{defn}
Recall that for strict algebras one could easily iterate structures
simply by considering $[\mathcal{P},[\mathcal{P},\mathcal{E}]]$ which
are classified by $\mathcal{P}\otimes\mathcal{P}$. Our journey into
weak algebras in a homotopy environment $\mathcal{E}$ led us to
the formation of the simplicial set $wAlg[\mathcal{P},\mathcal{E}]$
with the immediate drawback that we cannot, at least not in any straightforward
manner, iterate. This problem disappears in the dendroidal setting,
as we will see below, and is one of the technical advantages of dendroidal
sets over enriched operads in the study of weak algebraic structures.

\section{Dendroidal sets - a formalism for weak algebras}

We now return to non-enriched symmetric operads and introduce the
category of dendroidal sets, which is the natural category in which to define
nerves of symmetric operads. The category of dendroidal sets is a
presheaf category on the dendroidal category $\Omega$ and as such
one might expect it to be adequate only for the study of non-enriched
symmetric operads. However, we will see that it is in fact versatile
enough to treat enriched operads quite efficiently by means of the
homotopy coherent nerve construction. We do mention that for weak
algebraic structures in certain homotopy environments (such as differentially
graded vector spaces) dendroidal sets are inappropriate. One might
then consider dendroidal objects instead of dendroidal sets as is
explained in \cite{den set}, which also contains all of the results
below.

\subsection{The dendroidal category $\Omega$}

To define the dendroidal category $\Omega$ recall the definition
of symmetric rooted trees given above. It is evident that any such
tree $T$ can be thought of as a picture of a symmetric operad $\Omega(T)$:
The objects of $\Omega(T)$ are the edges of $T$ and the arrows are
freely generated by the vertices of $T$. In more detail, consider
the tree $T$ given by 
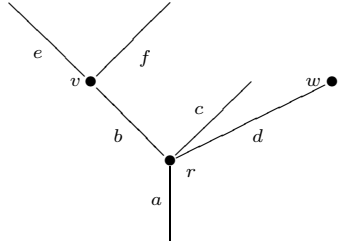
\[
\xymatrix{*{\,}\ar@{-}[dr]_{e} &  & *{\,}\ar@{-}[dl]^{f}\\
\,\ar@{}[r]|{\,\,\,\,\,\,\,\,\,\,\,\,\,\, v} & *{\bullet}\ar@{-}[dr]_{b} &  & *{\,}\ar@{-}[dl]_{c}\ar@{}[r]|{\,\,\,\,\,\,\,\,\,\,\,\, w} & *{\bullet}\ar@{-}[dll]^{d}\\
 &  & *{\bullet}\ar@{-}[d]_{a} & \,\ar@{}[l]^{r\,\,\,\,\,\,\,\,\,\,\,}\\
 &  & *{\,}}
\]
then $\Omega(T)$ has six objects, $a,b,\cdots,f$ and the following
generating operations: \[
r\in\Omega(T)(b,c,d;a),\]
\[
w\in\Omega(T)(-;d)\]
 and \[
v\in\Omega(T)(e,f;b).\]
The other operations are units (such as $1_{b}\in\Omega(T)(b;b)$),
arrows obtained freely by the $\Sigma_{n}$ actions, and formal compositions
of such arrows. 
\begin{defn}
Fix a countable set $X$. The \emph{dendroidal category} $\Omega$
has as objects all symmetric rooted trees $T$ whose edges $E(T)$
satisfy $E(T)\subseteq X$. The arrows $S\to T$ in $\Omega$ are
arrows $\Omega(S)\to\Omega(T)$ of symmetric operads. \end{defn}
\begin{rem}
The role of the set $X$ above should be thought of as the role variables
play in predicate calculus. The edges are only there to be carriers
of symbols and countably many such carriers will always be enough.
Of course, another choice of $X$ would result in an isomorphic category.
Note that the dendroidal category $\Omega$ is thus small (in fact
is itself countable). 
\end{rem}
Recall the linear trees $L_{n}$ and that the simplicial category
$\Delta$ is a skeleton of the category of finite linearly ordered
sets and order preserving maps. The subcategory of $\Omega$ spanned
by all trees of the form $L_{n}$ with $n\ge0$ is easily seen to
be equivalent to the simplicial category $\Delta$. In fact, $\Delta$
can be recovered from $\Omega$ in a more useful way.
\begin{prop}
(Slicing lemma for the dendroidal category) The simplicial category
$\Delta$ is obtained (up to equivalence) from the dendroidal category
$\Omega$ by slicing over the linear tree with one edge and no vertices:
$\Delta\cong\Omega/L_{0}$.
\end{prop}
We now describe several types of arrows that generate all of the arrows
in $\Omega$. Let $T$ be a tree and $v$ a vertex of valence 1 with
$in(v)=e$ and $out(v)=e'$. Consider the tree $T/v$, obtained from
$T$ by deleting the vertex $v$ and the edge $e'$, pictured locally
as\[
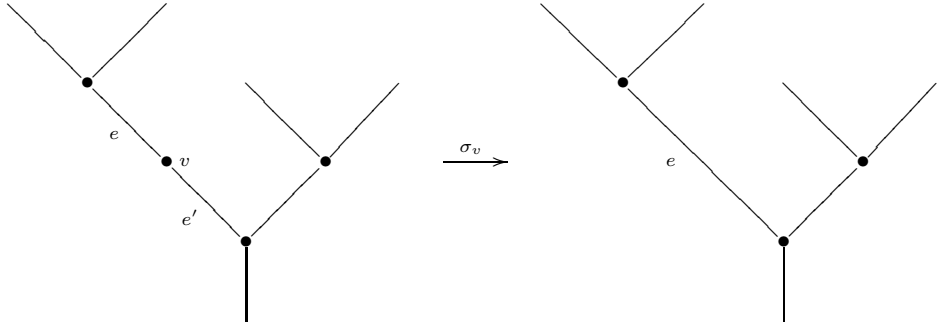
\begin{array}{ccc}
\xymatrix{*{\,}\ar@{-}[dr] &  & *{\,}\ar@{-}[dl]\\
 & *{\bullet}\ar@{-}[dr]_{e} &  & *{\,}\ar@{-}[dr] &  & *{\,}\ar@{-}[dl]\\
 &  & *{\bullet}\ar@{-}[dr]_{e'}\ar@{}|{\,\,\,\,\,\,\,\,\,\, v} &  & *{\bullet}\ar@{-}[dl]\\
 &  &  & *{\bullet}\ar@{-}[d]\\
 &  &  & *{\,}}
 & \xymatrix{\\\\\ar[r]^{\sigma_{v}} & *{}}
 & \xymatrix{*{\,}\ar@{-}[dr] &  & *{\,}\ar@{-}[dl]\\
 & *{\bullet}\ar@{-}[ddrr] &  & *{\,}\ar@{-}[dr] &  & *{\,}\ar@{-}[dl]\\
 & \,\ar@{}[r]|{\,\,\,\, e} &  &  & *{\bullet}\ar@{-}[dl]\\
 &  &  & *{\bullet}\ar@{-}[d]\\
 &  &  & *{\,}}
\end{array}\]
There is then a map in $\Omega$, denoted by $\sigma_{v}:T\rightarrow T/v$,
which sends $e$ and $e'$ in $T$ to $e$ in $T/v$. An arrow in
$\Omega$ of this kind is called a \emph{degeneracy. }

Consider now a tree $T$ and a vertex $v$ in $T$ with exactly one
inner edge attached to it. One can obtain a new tree $T/v$ by deleting
$v$ and all the outer edges attached to it to obtain, by inclusion
of edges, the arrow $\partial_{v}:T/v\to T$ in $\Omega$ called an
\emph{outer face.} For example, \[
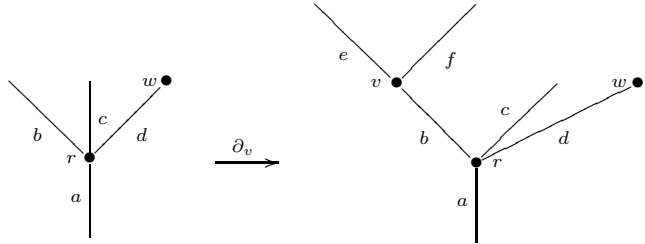
\begin{array}{ccc}
\xymatrix{\\*{\,}\ar@{-}[dr]_{b} & *{\,}\ar@{-}[d]^{c}\ar@{}[r]|{\,\,\,\,\,\,\,\,\,\,\,\, w} & *{\bullet}\ar@{-}[dl]^{d}\\
\,\ar@{}[r]|{\,\,\,\,\,\,\,\,\,\,\,\, r} & *{\bullet}\ar@{-}[d]_{a}\\
 & *{\,}}
 & \xymatrix{\\\\\ar[r]^{\partial_{v}} & *{}}
 & \xymatrix{*{\,}\ar@{-}[dr]_{e} &  & *{\,}\ar@{-}[dl]^{f}\\
\,\ar@{}[r]|{\,\,\,\,\,\,\,\,\,\,\, v} & *{\bullet}\ar@{-}[dr]_{b} &  & *{\,}\ar@{-}[dl]_{c}\ar@{}[r]|{\,\,\,\,\,\,\,\,\,\,\,\, w} & *{\bullet}\ar@{-}[dll]^{d}\\
 &  & *{\bullet}\ar@{-}[d]_{a} & \,\ar@{}[l]|{r\,\,\,\,\,\,\,\,\,\,\,}\\
 &  & *{\,}}
\end{array}\]
 and (to emphasize that it is sometimes possible to remove the root
of the tree $T$)\[
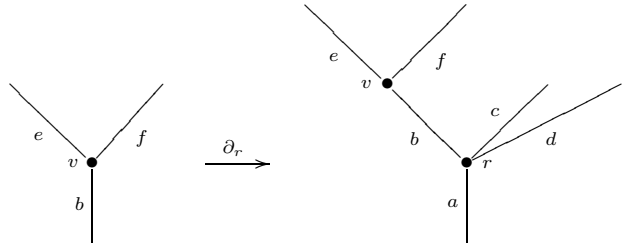
\begin{array}{ccc}
\xymatrix{\\*{\,}\ar@{-}[dr]_{e} &  & *{}\ar@{-}[dl]^{f}\\
\,\ar@{}[r]|{\,\,\,\,\,\,\,\,\,\,\,\, v} & *{\bullet}\ar@{-}[d]_{b}\\
 & *{\,}}
 & \xymatrix{\\\\\ar[r]^{\partial_{r}} & *{}}
 & \xymatrix{*{\,}\ar@{-}[dr]_{e} &  & *{\,}\ar@{-}[dl]^{f}\\
\,\ar@{}[r]|{\,\,\,\,\,\,\,\,\,\,\, v} & *{\bullet}\ar@{-}[dr]_{b} &  & *{\,}\ar@{-}[dl]_{c} & *{}\ar@{-}[dll]^{d}\\
 &  & *{\bullet}\ar@{-}[d]_{a} & \,\ar@{}[l]|{r\,\,\,\,\,\,\,\,\,\,\,}\\
 &  & *{\,}}
\end{array}\]
are both outer faces. 

Given a tree $T$ and an inner edge $e$ in $T$, one can obtain a
new tree $T/e$ by contracting the edge $e$. One then obtains, by
inclusion of edges, the map $\partial_{e}:\Omega(T/e)\rightarrow\Omega(T)$
in $\Omega$ called an \emph{inner face.} For example,

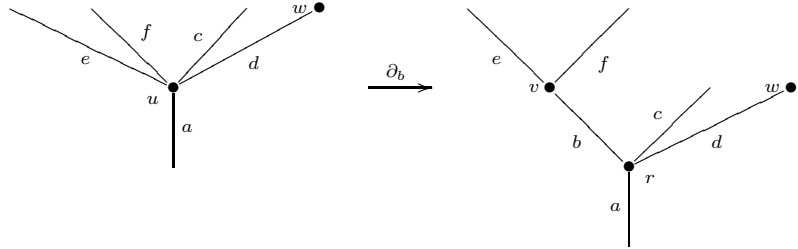
\[
\begin{array}{ccc}
\xymatrix{*{\,}\ar@{-}[rrd]_{e} & *{\,}\ar@{-}[rd]^{f} &  & *{\,}\ar@{-}[dl]_{c}\ar@{}[r]|{\,\,\,\,\,\,\,\,\, w} & *{\bullet}\ar@{-}[lld]^{d}\\
 & \,\ar@{}[r]_{\,\,\,\,\,\,\,\,\,\,\, u} & *{\bullet}\ar@{-}[d]^{a}\\
 &  & *{\,}}
 & \xymatrix{\\\ar[r]^{\partial_{b}} & *{}}
 & \xymatrix{*{\,}\ar@{-}[dr]_{e} &  & *{\,}\ar@{-}[dl]^{f}\\
\,\ar@{}[r]|{\,\,\,\,\,\,\,\,\,\,\,\,\,\, v} & *{\bullet}\ar@{-}[dr]_{b} &  & *{\,}\ar@{-}[dl]_{c}\ar@{}[r]|{\,\,\,\,\,\,\,\,\,\,\,\, w} & *{\bullet}\ar@{-}[dll]^{d}\\
 &  & *{\bullet}\ar@{-}[d]_{a} & \,\ar@{}[l]^{r\,\,\,\,\,\,\,\,\,\,\,}\\
 &  & *{\,}}
\end{array}\]

\begin{thm}
Any map $\xymatrix{T\ar[r]^{f} & T'}
$ in $\Omega$ factors uniquely as $f=\varphi\pi\delta$, where $\delta$
is a composition of degeneracy maps, $\pi$ is an isomorphism, and
$\varphi$ is a composition of (inner and outer) face maps. 
\end{thm}
This result generalizes the familiar simplicial relations in the definition
of a simplicial set.

\subsubsection{The category of dendroidal sets}
\begin{defn}
The category of \emph{dendroidal sets }is the presheaf category $dSet=Set_{\Omega}$.
Thus a dendroidal set $X$ consists of a collection of sets $\{X_{T}\}_{T\in\Omega_{0}}$
together with various maps between them. An element $x\in X_{T}$
is called a \emph{dendrex of shape $T$, }or a $T$-\emph{dendrex}. 
\end{defn}
For each tree $T\in\Omega_{0}$ there is associated the \emph{representable
dendroidal set }$\Omega[T]=\Omega(-,T)$ which, by the Yoneda Lemma,
serves to classify $T$ dendrices in $X$ via the natural bijection
$X_{T}\cong dSet(\Omega[T],X)$. The functor $\Omega\to Ope$ which
sends $T$ to $\Omega(T)$ induces an adjunction $\xymatrix{dSet\ar@<2pt>[r]^{\tau_{d}} & Ope\ar@<2pt>[l]^{N_{d}}}
$, of which $N_{d}$, called the \emph{dendroidal nerve functor}, is
given explicitly, for a symmetric operad $\mathcal{P}$, by \[
N_{d}(\mathcal{P})_{T}=Ope(\Omega(T),\mathcal{P}).\]
For linear trees $L_{n}$ we write somewhat ambiguously $X_{n}$ instead
of $X_{L_{n}}$. This is a harmless convention since for any two linear
trees $L_{n}$ and $L_{n}^{'}$ there is a \emph{unique} isomorphism
$L_{n}\to L_{n}^{'}$ in $\Omega$. Consider the dendroidal set $\star=\Omega[L_{0}]$.
\begin{lem}
(Slicing lemma for dendroidal sets) There is an equivalence of categories
$dSet/\star\cong sSet$. If we identify $sSet$ as a subcategory
of $dSet$ then the forgetful functor $i_{!}:sSet\to dSet$ has a
right adjoint $i^{*}$ which itself has a right adjoint $i_{*}$. \end{lem}
\begin{proof}
We omit the details and just remark that the adjunctions mentioned
can be obtained (equivalently) in one of two ways. The first is to
consider $\Delta$ as a subcategory of $\Omega$ via an embedding
functor $i:\Delta\to\Omega$. This functor $i$ then induces a functor
$i^{*}:dSet\to sSet$ which, from the general theory of presheaf categories
(see e.g., \cite{Mac Moer Sheaves}), has a left adjoint $i_{!}$
and a right adjoint $i_{*}$. The second way to obtain the adjunctions
is to use Remark \ref{rem:Slicing}, with $\mathcal{C}=dSet$ and
$A=\Omega[L_{0}]$. \end{proof}
\begin{prop}
Slicing the adjunction $\xymatrix{dSet\ar@<2pt>[r]^{\tau_{d}} & Ope\ar@<2pt>[l]^{N_{d}}}
$ over $\star$ gives the usual adjunction $\xymatrix{sSet\ar@<2pt>[r]^{\tau} & Cat\ar@<2pt>[l]^{N}}
$ with $N$ the nerve functor and $\tau$ the fundamental category
functor.\end{prop}
\begin{proof}
The precise meaning of the statement is that denoting a one-object
operad with just the identity arrow again by $\star$ and for the
dendroidal set $\star=\Omega[L_{0}]$ one has, by slight abuse of
notation, that $N_{d}(\star)=\star$ and $\tau_{d}(\star)=\star$; thus the functors $N_{d}$ and $\tau_{d}$ restrict to the
respective slices $dSet/\star$ and $Ope/\star$. Then under the identifications
$sSet\cong dSet/\star$ and $Cat\cong Ope/\star$ these restrictions
give the nerve functor $N:Cat\to sSet$ and its left adjoint $\tau$. 
\end{proof}
A general rule of thumb is that any definition or theorem of dendroidal
sets will yield, by slicing over $\star=\Omega[L_{0}]$, a corresponding
definition or theorem of simplicial sets. A similar principle is true
for operads and categories. We will loosely refer to this process
as 'slicing' and say, in the example above for instance, that the
usual nerve functor of categories is obtained by slicing the dendroidal
nerve functor.
\begin{defn}
Let $X$ and $Y$ be two dendroidal sets. Their tensor product is
given by the colimit\[
X\otimes Y=\lim_{\Omega[T]\rightarrow X,\Omega[S]\to Y}N_{d}(\Omega(T)\otimes\Omega(S)),\]
Here we use the canonical expression of a presheaf as a colimit of representables.
\end{defn}

As an example of our convention about slicing we mention that slicing
the tensor product of dendroidal sets yields the cartesian product
of simplicial sets. Note however, that the tensor product in $dSet$
is not the cartesian product.
\begin{thm}
The category $dSet$ with the tensor product defined above is a closed
monoidal category. \end{thm}
\begin{proof}
This follows by general abstract nonsense. The internal hom is given
for two dendroidal sets $X$ and $Y$ by \[
[X,Y]_{T}=dSet(X\otimes\Omega[T],Y).\]

\end{proof}
Slicing this theorem proves that $sSet$ is cartesian closed with
the usual formula for the internal hom. 
\begin{thm}
In the diagram\[
\xymatrix{Cat\ar@<2pt>[r]^{j_{!}\,\,\,}\ar@<2pt>[d]^{N} & Ope\ar@<2pt>[l]^{j^{*}\,\,\,}\ar@<2pt>[d]^{N_{d}}\\
sSet\ar@<2pt>[r]^{i_{!}}\ar@<2pt>[u]^{\tau} & dSet\ar@<2pt>[l]^{i^{*}}\ar@<2pt>[u]^{\tau_{d}}}
\]
all pairs of functors are adjunctions with the left adjoint on top
or to the left. Furthermore, the following canonical commutativity
relations hold:
\begin{eqnarray*}
 &  & \tau N\cong id\\
 &  & \tau_{d}N_{d}\cong id\\
 &  & i^{*}i_{!}\cong id\\
 &  & j^{*}j_{!}\cong id\\
 &  & j_{!}\tau\cong\tau_{d}i_{!}\\
 &  & Nj^{*}\cong i^{*}N_{d}\\
 &  & i_{!}N\cong N_{d}j_{!}.\end{eqnarray*}
If we consider the cartesian structures on $Cat$ and $sSet$, the
Boardman-Vogt tensor product on $Ope$, and the tensor product of
dendroidal sets then the four categories are symmetric closed monoidal
categories and the functors $i_{!},N,\tau,j_{!}$ and
$\tau_{d}$ are strong monoidal.\end{thm}
\begin{rem}
The dendroidal nerve functor $N_{d}$ is not monoidal, a fact that
plays a vital role in the applicability of dendroidal sets to iterated
weak algebraic structures, as we will see below. 
\end{rem}
We do have the following property.
\begin{prop}
\label{pro:tau of tensor}For symmetric operads $\mathcal{P}$ and
$\mathcal{Q}$ there is a natural isomorphism \[
\tau_{d}(N_{d}(\mathcal{P})\otimes N_{d}(\mathcal{Q}))\cong\mathcal{P}\otimes\mathcal{Q}.\]
\end{prop}
\begin{lem}
The dendroidal nerve functor commutes with internal Homs in the sense
that for any two operads $\mathcal{P}$ and $\mathcal{Q}$ we have\[
N_{d}([\mathcal{P},\mathcal{Q}])\cong[N_{d}(\mathcal{P}),N_{d}(\mathcal{Q})].\]
Moreover, for simplicial sets $X$ and $Y$ we have\[
[i_{!}(X),i_{!}(Y)]\cong i_{!}([X,Y]).\]

\end{lem}
The proofs of these results are not hard.

\subsection{Algebras in the category of dendroidal sets}

We again introduce a syntactic difference between dendroidal sets
thought of as encoding structure and dendroidal sets as environments
to interpret structures in. 
\begin{defn}
Let $E$ and $X$ be dendroidal sets. The dendroidal set $[X,E]$
is called the dendroidal set of $X$-\emph{algebras} in $E$. An element
in $[X,E]_{L_{0}}$ is called an $X$-\emph{algebra} in $E$. An element
of $[X,E]_{L_{1}}$ is called a \emph{map of $X$-algebras in $E$.} 
\end{defn}
Let us first note that this definition extends the notion of $\mathcal{P}$-algebras
in $\mathcal{E}$ for symmetric operads in the sense that for symmetric operads $\mathcal{P}$ and $\mathcal{E}$ there is a natural isomorphism
\[
[N_{d}(\mathcal{P}),N_{d}(\mathcal{E})]\cong N_{d}([\mathcal{P},\mathcal{E}]).\]
Indeed, this is just the statement that $N_{d}$ commutes with internal Homs. 

We thus see that the dendroidal nerve functor embeds $Ope$ in $dSet$
in such a way that the notion of algebras is retained and in both
cases is internalized in the form of an internal Hom with respect
to a suitable tensor product. We now wish to study homotopy invariance
of algebra structures in a dendroidal set $E$. The first step is
to specify those arrows along which such algebras are to be invariant. 

Recall that for symmetric operads the diagram\[
\xymatrix{0\to1\ar[r]^{\,\,\,\,\, f}\ar[d] & \mathcal{P}\\
0\leftrightarrows1\ar@{..>}[ru]_{(f,g)}}
\]
(where the vertical arrow is the inclusion of the free-living arrow
into the free-living isomorphism) admits a lift precisely when the
arrow $f$ admits an inverse $g$. Taking the dendroidal nerve of
this diagram and replacing in it $N_{d}(\mathcal{P})$ by an arbitrary
dendroidal set $X$ we arrive at the following definition. 
\begin{defn}
Let $X$ be a dendroidal set. An \emph{equivalence }is a dendrex $x:\Omega[L_{1}]\to X$
such that in the diagram\[
\xymatrix{\Omega[L_{1}]\ar[r]^{x}\ar[d] & X\\
N_{d}(0\leftrightarrows1)\ar@{..>}[ru]_{\hat{x}}}
\]
a lift $\hat{x}$ exists. 
\end{defn}
Note, that since $\Omega[L_{1}]\cong i_{!}(\Delta[1])$ and $N_{d}(0\leftrightarrows1)=i_{!}(N(0\leftrightarrows1))$,
by adjunction the dendrex $x:\Omega[L_{1}]\to X$ is an equivalence
if, and only if, in the corresponding diagram of simplicial sets\[
\xymatrix{\Delta[1]\ar[d]\ar[r]^{x} & i^{*}(X)\\
N(0\leftrightarrows1)\ar@{..>}[ru]}
\]
a lift exists. Thus, being a weak equivalence in the dendroidal set
$X$ is actually a property of the simplicial set $i^{*}(X)$. 
\begin{rem}
Note that $N(0\leftrightarrows1)$ is the simplicial infinite dimensional
sphere $S^{\infty}$ and thus a lift $S^{\infty}\to i^{*}(X)$ is
a rather complicated object. Intuitively, it is a coherent choice
of a homotopy inverse of the simplex $x:\Delta[1]\to i^{*}(X)$, together
with coherent choices of homotopies, homotopies between homotopies,
etc. See \cite{quasi cat} for more details. Note moreover, that an
equivalence in $X$ is in some sense as weak as $X$ would allow it
to be. If $X=N_{d}(\mathcal{P})$ then a dendrex $\Omega[L_{1}]\to X$
is an equivalence if, and only if, the corresponding unary arrow in
$\mathcal{P}$ is an isomorphism. We will see below a more refined
nerve construction in which equivalences correspond to a notion weaker
than isomorphism. 
\end{rem}
We can now formulate the homotopy invariance property in the language
of dendroidal sets. Let $X$ and $E$ be dendroidal sets. We identify,
somewhat ambiguously, the set $X_{\eta}$ with the dendroidal set $\coprod_{x\in X_{\eta}}\Omega[\eta]$.
Then there is a map of dendroidal sets $X_{\eta}\to X$ which induces
a mapping $[X,E]\to[X_{\eta},E]$. Consider now a family $\{f_{x}\}_{x\in X_{\eta}}$
where each $f_{x}$ is an equivalence in $X$. Then this family can
be extended (usually in many different ways) to give a map $\hat{f}:N_{d}(0\leftrightarrows1)\to[X_{\eta},E]$. 
\begin{defn}
\label{def:homotop inv proper}Let $X$ and $E$ be a dendroidal sets.
We say that $X$-algebras in $E$ have the \emph{homotopy invariance
property} if for every $X$-algebra in $E$, given by $F:X\to E$,
and any family $\{f_{x}\}_{x\in X_{\eta}}$ and any extension of it to
$\hat{f}$ as above that fit into the commutative diagram\[
\xymatrix{\Omega[\eta]\ar[d]\ar[rrr]^{\forall F} &  &  & [X,E]\ar[d]\\
N_{d}(0\leftrightarrows1)\ar[rrr]^{\hat{f}}\ar@{..>}[rrru]^{\exists\alpha} &  &  & [X_{\eta},E]}
\]
a lift $\alpha$ exists. 
\end{defn}
Intuitively, the lift $\alpha$ consists of two $X$-algebras in $E$,
the first being $F$ and the second one being obtained by transferring
the $X$-algebra structure given by $F$ along the equivalences $\{f_{x}\}_{x\in X_{\eta}}$.

\subsection{The homotopy coherent nerve and weak algebras}

We now show how dendroidal sets enter the picture in the context of
operads enriched in a symmetric closed monoidal model category $\mathcal{E}$
with a chosen interval object (which we call a homotopy environment).
Recall then that the Berger-Moerdijk generalization of the Boardman-Vogt
$W$-construction sends a symmetric operad $\mathcal{P}$ enriched in
$\mathcal{E}$ to a cofibrant replacement $W\mathcal{P}$. Recall
as well that any non-enriched symmetric operad can be seen as a discrete
symmetric operad enriched in $\mathcal{E}$. 
\begin{defn}
Fix a homotopy environment $\mathcal{E}$. Given a symmetric operad
$\mathcal{P}$ enriched in $\mathcal{E}$ its \emph{homotopy coherent
dendroidal nerve }is the dendroidal set whose set of $T$-dendrices
is \emph{\[
hcN_{d}(\mathcal{P})_{T}=Ope(\mathcal{E})(W(\Omega(T)),\mathcal{P})\]
}of $\mathcal{E}$-enriched functors between $\mathcal{E}$-enriched
operads, where $\Omega(T)$ is seen as a discrete operad enriched in
$\mathcal{E}$.
\end{defn}
The homotopy coherent dendroidal nerve construction, together with
the closed monoidal structure on $dSet$ given above, allows for the
internalization of the notion of weak algebras. We illustrate this: 
\begin{defn}
Let $\mathcal{P}$ be a non-enriched symmetric operad and $\mathcal{E}$
a homotopy environment. The dendroidal set $[N_{d}(\mathcal{P}),hcN_{d}(\hat{\mathcal{E}})]$
is called the dendroidal set of \emph{weak $\mathcal{P}$-algebras
in $\mathcal{E}$}. Here we view $\mathcal{E}$ as an operad enriched
in itself (since $\mathcal{E}$ is assumed closed) and thus $\hat{\mathcal{E}}$
as a symmetric operad enriched in $\mathcal{E}$ is well-defined.
\end{defn}
It can be shown that the $L_{0}$ dendrices in $[N_{d}(\mathcal{P}),hcN_{d}(\hat{\mathcal{E}})]$
correspond to symmetric operad maps $W(\mathcal{P})\to\hat{\mathcal{E}}$
and thus are weak $\mathcal{P}$-algebras. Moreover, the $L_{1}$
dendrices can be seen to correspond to symmetric operad maps $W(\mathcal{P}\otimes[1])\to\hat{\mathcal{E}}$
and thus are weak maps of weak $\mathcal{P}$-algebras. We have thus
recovered an internalization of weak algebras and their weak maps
and can now consider iterated weak algebraic structures completely
analogously to the way this can be done in the context of non-enriched
symmetric operads. We illustrate how this works in two examples below.

\subsection{Application to the study of $A_{\infty}$-spaces and weak $n$-categories}

Recall that an $A_{\infty}$-space is an algebra for the topologically
enriched operad $W(As)$, where $As$ is the non-enriched symmetric
operad that classifies monoids. Let $A=N_{d}(As)$; then, by definition, $[A,hcN_{d}(Top)]$ is the dendroidal set of $A_{\infty}$-spaces
and their weak (multivariable) mappings. In the classical definition
of $A_{\infty}$-spaces it is not at all clear how to define $n$-fold
$A_{\infty}$-spaces. However, we now have a perfectly natural such
definition. 
\begin{defn}
The dendroidal set $nA_{\infty}$ of \emph{$n$-fold $A_{\infty}$-spaces}
is defined recursively as follows. For $n=1$ we set $1A_{\infty}=[A,hcN_{d}(Top)]$
and for $n\ge 1$: $(n+1)A_{\infty}=[A,nA_{\infty}]$.
\end{defn}
Thus, we obtain at once notions of weak multivariable mappings of
$n$-fold $A_{\infty}$-spaces. And, since $dSet$ is closed monoidal,
we can immediately classify $n$-fold $A_{\infty}$-spaces.
\begin{prop}
For any $n\ge1$ the dendroidal set $A^{\otimes n}$ classifies $n$-fold
$A_{\infty}$-spaces. 
\end{prop}
It is at this point not known exactly how $n$-fold $A_{\infty}$-spaces
relate to $n$-fold loop spaces. However, the recent work \cite{Fiedor} of Fiedorowicz and Vogt on interchanging $A_\infty$ and $E_{n}$ structures is a first step towards a full comparison of the dendroidal and classical approaches. 
\begin{rem}
Note that were the dendroidal nerve functor monoidal our definition
of $n$-fold $A_{\infty}$-spaces would stabilize at $n=2$. Indeed, we would then have $A^{\otimes n}=N_{d}(As)^{\otimes n}=N_{d}(As^{\otimes n})=N_{d}(Comm)$. 
\end{rem}
A similar application, but technically slightly more complicated,
is to obtain an iterative definition of weak $n$-categories. First
notice that the fact that categories, as well as symmetric operads,
can be enriched in a symmetric monoidal category $\mathcal{E}$ is
a consequence of the ability to in fact enrich in an arbitrary symmetric
operad. We leave the details of defining what a category (or operad)
enriched in a symmetric operad $\mathcal{E}$ is to the reader and only mention
that this is related to the idea of enriching in an $fc$-multicategory
(see \cite{Leins gen enrich}). We now show how in fact categories
(and operads) can be enriched in a dendroidal set. Recall from Example
\ref{exa:classifying categories over A} that for any set $A$ there
is a symmetric operad $C_{A}$ that classifies categories over $A$. 

Once more, the ability to easily iterate within the category of dendroidal
sets naturally leads to a definition of weak $n$-categories enriched
in a dendroidal set $X$ as follows. 
\begin{defn}
Let $X$ be a dendroidal set. The dendroidal set $[N_{d}(\mathcal{C}_{A}),X]$
is called the dendroidal set of \emph{categories over $A$ enriched
in $X$} and is denoted by $Cat(X)_{A}.$ 
\end{defn}
It can easily be verified that enriching in the dendroidal nerve of
$\hat{\mathcal{E}}$ for $\mathcal{E}$ a symmetric monoidal category
agrees with the notion of enrichment in the usual sense. 

At this point we would like to collate the various dendroidal sets
$Cat(X)_{A}$ into a single dendroidal sets. There is here a technical
difficulty and so as not to interrupt the flow of the presentation
we refer the reader to Section 4.1 of \cite{den set} for the details
of the construction. One then obtains the dendroidal set $Cat(X)$
of categories enriched in $X$. Similarly, using the operad $O_{A}$
classifying symmetric operads over $A$, we can obtain the dendroidal
set $Ope(X)$ of symmetric operads enriched in $X$.
\begin{defn}
Let $X$ be a dendroidal set. Let $_{0}Cat(X)=X$ and define recursively
$_{n+1}Cat(X)=Cat(_{n}Cat(X))$ for each $n\ge1$. We call $_{n}Cat(X)$
the dendroidal set of \emph{$n$-categories enriched in $X$. }
\end{defn}
In particular, considering the category $Cat$ with its categorical
model structure and taking $X=hcN_{d}(Cat)$ we obtain for each $n\ge0$
the dendroidal set $_{n}Cat={}_{n}Cat(X)$, which we call the dendroidal
set of weak $n$-categories. In \cite{Lukacs thesis} the dendrices
in $_{2}Cat(X)_{\eta}$ and $_{3}Cat(X)_{\eta}$ are compared with
other definitions of weak $2$-categories and weak $3$-categories
to show that the notions are in fact equivalent. The complexity of
such comparisons increases rapidly with $n$ and is currently, as
is the case with many definitions of weak $n$-categories (see \cite{Leinster survey}
for a survey of such), not settled. We mention that we can also consider
categories weakly enriched in other dendroidal sets such as $hcN_{d}(Top)$
or $hcN_{d}(sSet)$, where again a full comparison with existing structures
is yet to be completed. Of course, we can also consider weak $n$-operads
of various sorts as well. 

We conclude this section by considering the Baez-Breen-Dolan stabilization
hypothesis for weak $n$-categories as defined above. With every reasonable
definition of weak $n$-categories there is usually associated a notion
of $k$-monoidal $n$-categories for every $k\ge0$. These are weak
$(n+k)$-categories having trivial information in all dimensions up
to and including $k$. The stabilization hypothesis is that for fixed
$n$ the complexity of these structures stabilizes at $k=n+2$. Given
a concrete definition of weak $n$-categories this hypothesis can
be made exact and it becomes a conjecture. In our case we proceed
as follows. 
\begin{defn}
Let $n\ge0$ be fixed. For $k\ge0$ we define recursively the dendroidal
set $wCat_{k}^{n}$ of \emph{weak $k$-monoidal $n$-categories} as
follows. For $k=0$ we set $wCat_{0}^{n}=\ _{n}Cat$ and for $k>0$
we define $wCat_{k}^{n}=[A,wCat_{k-1}^{n}]$, where $A=N_{d}(As)$.
A dendrex of shape $\eta$ in $wCat_{k}^{n}$ is called a \emph{$k$-monoidal
$n$-category}.\end{defn}
\begin{conjecture}
\label{con:(The-Baez-Breen-Dolan-stabilization}(The Baez-Breen-Dolan
stabilization hypothesis for our notion of $n$-categories) For a
fixed $n\ge0$, there is an isomorphism of dendroidal sets between
$wCat_{k}^{n}$ and $wCat_{n+2}^{n}$ for any $k\ge n+2$. 
\end{conjecture}

\section{Dendroidal sets - models for $\infty$-operads}

In this section we show how dendroidal sets are used to model $\infty$-operads.
As very brief motivation for the concepts to follow we first discuss
$\infty$-categories, then we present that part of the theory of dendroidal
sets needed to define the Cisinski-Moerdijk model structure on dendroidal
sets which establishes dendroidal sets as models for homotopy operads
and illustrate some of its consequences. The proofs of the results
below can be found in \cite{dSet model hom op,inn Kan in dSet}.

\subsection{$\infty$-categories briefly}

We have seen above that in general, weak $\mathcal{P}$-algebras in
a homotopy environment $\mathcal{E}$ and their weak maps fail to
form a category and that in fact one is immediately led to define
the simplicial set of weak algebras $wAlg[\mathcal{P},\mathcal{E}]$.
The failure of this simplicial set to be the nerve of a category is
a reflection of composition not being associative. However, the composition
of weak maps is associative up to coherent homotopies, a fact which
induces some extra structure on the simplicial set $wAlg[\mathcal{P},\mathcal{E}]$.
Boardman and Vogt in \cite{BV book} formulated this extra structure by means of a condition called the restricted Kan condition. To define it recall that
a horn in a simplicial set $X$ is a mapping $\Lambda^{k}[n]\to X$, where $\Lambda^{k}[n]$ is the union in $\Delta[n]$ of all faces
except the one opposite the vertex $k$. A horn is called \emph{inner
}when $0<k<n$. Boardman and Vogt in \cite{BV book}, page 102, define
\begin{quote}
A simplicial set $X$ is said to satisfy the restricted Kan condition
if every inner horn $\Lambda^{k}[n]\to X$ has a filler.
\end{quote}
and so $\infty$-categories were born. They consequently prove that
in the context of topological operads the simplicial set of weak algebras
satisfies the restricted Kan condition. Simplicial sets satisfying
the restricted Kan condition are extensively studied by Joyal (in
\cite{quasi cat,quasi cat book} under the name 'quasicategories')
and by Lurie (in e.g., \cite{Lurie} under the name '$(\infty,1)$-categories'
or more simply '$\infty$-categories').

There are several ways to model $\infty$-categories, of which the
above restricted Kan condition is one. Three other models are complete
Segal spaces, Segal categories, and simplicial categories. 
For each of these models there is an appropriate Quillen model structure
rendering the four different models Quillen equivalent (see \cite{surv infty cat}
for a detailed survey). By considering dendroidal sets instead of
simplicial sets Cisinski and Moerdijk in \cite{den Seg sp,dSet and simp ope}
introduce the analogous dendroidal notions: complete dendroidal Segal
spaces, Segal operads, and simplicial operads. Moreover, they establish
Quillen model structures for each of these notions, proving they are
all Quillen equivalent to a Quillen model structure on dendroidal
sets they establish in \cite{dSet model hom op}. All of these model
structures and equivalences, upon slicing over a suitable object,
reduce to the equivalence of the simplicial based structures mentioned
above. 

There is yet another approach to $\infty$-operads, taken by Lurie
\cite{high top th}, which defines an $\infty$-operad to be a simplicial
set with extra structure. In Lurie's approach the highly developed
theory of simplicial sets and quasicategories is readily available
to provide a rich theory of $\infty$-operads. However, the extra
structure that makes a simplicial set into an $\infty$-operad is
quite complicated, rendering working with explicit examples of $\infty$-operads
difficult. The approach via dendroidal sets replaces the relative
simplicity of the combinatorics of linear trees by the complexity
of the combinatorics of trees which renders existing simplicial theory
unusable but offers very many explicit examples of $\infty$-operads.
We believe that this trade-off in complexity will result in these
two approaches mutually enriching each other as a future comparisons
unfold. 

Below, following \cite{dSet model hom op} we give a short presentation
of the approach to $\infty$-operads embodied in the Cisinski-Moerdijk
model structure on $dSet$ that slices to the Joyal model structure
on $sSet$ and use this model structure to prove a homotopy invariance
property for algebras in $dSet$. From this point on $\infty$-category
means a quasicategory.

\subsection{Horns in $dSet$}

We first introduce some concepts needed for the definition referring
the reader to \cite{den set,inn Kan in dSet} for more details. 
\begin{defn}
Let $T$ be a tree and $\alpha:S\rightarrow T$ a face map in $\Omega$.
The \emph{$\alpha$-face} of $\Omega[T]$, denoted by $\partial_{\alpha}\Omega[T]$,
is the dendroidal subset of $\Omega[T]$ which is the image of the
map $\Omega[\alpha]:\Omega[S]\rightarrow\Omega[T]$. Thus we have
that \[
\partial_{\alpha}\Omega[T]_{R}=\{\xymatrix{R\ar[r] & S\ar[r]^{\alpha} & T}
\mid R\rightarrow S\in\Omega[S]_{R}\}.\]
When $\alpha$ is obtained by contracting an inner edge $e$ in $T$
we denote $\partial_{\alpha}$ by $\partial_{e}$. 

Let $T$ be a tree. The \emph{boundary} of $\Omega[T]$ is the dendroidal
subset $\partial\Omega[T]$ of $\Omega[T]$ obtained as the union
of all the faces of $\Omega[T]$:

\[
\partial\Omega[T]=\bigcup_{\alpha\in\Phi_{1}(T)}\partial_{\alpha}\Omega[T].\]
where $\Phi_{1}(T)$, is the set of all faces of $T$. 
\end{defn}

\begin{defn}
Let $T$ be a tree and $\alpha\in\Phi_{1}(T)$ a face of $T$. The
\emph{$\alpha$-horn} in $\Omega[T]$ is the dendroidal subset $\Lambda^{\alpha}[T]$
of $\Omega[T]$ which is the union of all the faces of $T$ except
$\partial_{\alpha}\Omega[T]$:\[
\Lambda^{\alpha}[T]=\bigcup_{\beta\ne\alpha\in\Phi_{1}(T)}\partial_{\beta}\Omega[T].\]

The horn is called an \emph{inner horn} if $\alpha$ is an inner face,
otherwise it is called an \emph{outer horn}. We will denote an inner
horn $\Lambda^{\alpha}[T]$ by $\Lambda^{e}[T]$, where $e$ is the
contracted inner edge in $T$ that defines the inner face $\alpha=\partial_{e}:\Omega[T/e]\rightarrow\Omega[T]$.
A horn in a dendroidal set $X$ is a map of dendroidal sets $\Lambda^{\alpha}[T]\rightarrow X$.
It is inner (respectively outer) if the horn $\Lambda^{\alpha}[T]$
is inner (respectively outer). \end{defn}
\begin{rem}
It is trivial to verify that these notions for dendroidal sets extend
the common ones for simplicial sets in the sense, for example, that
for the simplicial horn $\Lambda^{k}[n]\subseteq\Delta[n]$, the dendroidal
set \[
i_{!}(\Lambda^{k}[n])\subseteq i_{!}(\Delta[n])=\Omega[L_{n}]\]
is a horn in the dendroidal sense. Furthermore, the horn $\Lambda^{k}[n]$
is inner (i.e., $0<k<n$) if, and only if, the horn $i_{!}(\Lambda^{k}[n])$
is inner. 
\end{rem}
Both the boundary $\partial\Omega[T]$ and the horns $\Lambda^{\alpha}[T]$
in $\Omega[T]$ can be described as colimits as follows. 
\begin{defn}
Let $T_{1}\rightarrow T_{2}\rightarrow\cdots\rightarrow T_{n}$ be
a sequence of $n$ face maps in $\Omega$. We call the composition
of these maps a \emph{subface} of $T_{n}$ of \emph{codimension}
$n$. \end{defn}
\begin{prop}
\label{pro:Sub-faceOfASub-Face}Let $S\rightarrow T$ be a subface
of $T$ of codimension $2$. The map $S\rightarrow T$ decomposes
in precisely two different ways as a composition of faces. 
\end{prop}
Let $\Phi_{i}(T)$ be the set of all subfaces of $T$ of codimension
$i$. The proposition implies that for each $\beta:S\rightarrow T\in\Phi_{2}(T)$
there are precisely two face maps $\beta_{1}:S\rightarrow T_{1}$
and $\beta_{2}:S\rightarrow T_{2}$ that factor $\beta$ as a composition
of face maps. Using these maps we can form two maps $\gamma_{1}$
and $\gamma_{2}$\[
\coprod_{S\rightarrow T\in\Phi_{2}(T)}\Omega[S]\rightrightarrows\coprod_{R\rightarrow T\in\Phi_{1}(T)}\Omega[R]\]
where $\gamma_{i}$ ($i=1,2$) has component$\xymatrix{\Omega[S]\ar[r]^{\Omega[\beta_{i}]} & \Omega[T_{i}]\ar[r] & \coprod\Omega[R]}
$ for each $\beta:S\rightarrow T\in\Phi_{2}(T)$. 
\begin{lem}
Let $T$ be a tree in $\Omega$. With notation as above we have that
the boundary $\partial\Omega[T]$ is a coequalizer\[
\coprod_{S\rightarrow T\in\Phi_{2}(T)}\Omega[S]\rightrightarrows\coprod_{R\rightarrow T\in\Phi_{1}(T)}\Omega[R]\rightarrow\partial\Omega[T]\]
of the two maps $\gamma_{1},\gamma_{2}$ constructed above. \end{lem}
\begin{cor}
A map of dendroidal sets $\partial\Omega[T]\rightarrow X$ corresponds
exactly to a sequence $\{x_{R}\}_{R\rightarrow T\in\Phi_{1}(T)}$
of dendrices whose faces match, in the sense that for each subface
$\beta:S\rightarrow T$ of codimension $2$ we have $\beta_{1}^{*}(x_{T_{1}})=\beta_{2}^{*}(x_{T_{2}})$.
\end{cor}
A similar presentation for horns holds as well. For a fixed face $\alpha:S\rightarrow T\in\Phi_{1}(T)$
consider the parallel arrows defined by making the following diagram
commute\[
\xymatrix{\Omega[S]\ar[d]\ar[r]^{\beta_{1}} & \Omega[T_{1}]\ar[d]\\
\coprod_{\beta:S\rightarrow T\in\Phi_{2}(T)}\Omega[S]\ar@<2pt>[r]\ar@<-2pt>[r] & \coprod_{R\rightarrow T\ne\alpha\in\Phi_{1}(T)}\Omega[R]\\
\Omega[S]\ar[u]\ar[r]^{\beta_{2}} & \Omega[T_{2}]\ar[u]}
\]
where the vertical arrows are the canonical injections into the coproduct
and where we use the same notation as above. 
\begin{lem}
Let $T$ be a tree in $\Omega$ and $\alpha$ a face of $T$. In the
diagram \[
\coprod_{S\rightarrow T\in\Phi_{2}(T)}\Omega[S]\rightrightarrows\coprod_{R\rightarrow T\ne\alpha\in\Phi_{1}(T)}\Omega[R]\rightarrow\Lambda^{\alpha}[T]\]
 the dendroidal set $\Lambda^{\alpha}[T]$ is the coequalizers of
the two maps constructed above. \end{lem}
\begin{cor}
A horn $\Lambda^{\alpha}[T]\rightarrow X$ in $X$ corresponds exactly
to a sequence $\{x_{R}\}_{R\rightarrow T\ne\alpha\in\Phi_{1}(T)}$
of dendrices that agree on common faces in the sense that if $\beta:S\rightarrow T$
is a subface of codimension $2$ which factors as\[
\xymatrix{ & R_{1}\ar[rd]^{\alpha_{1}}\\
S\ar[rr]^{\beta}\ar[ru]^{\beta_{1}}\ar[rd]_{\beta_{2}} &  & T\\
 & R_{2}\ar[ru]_{\alpha_{2}}}
\]
then $\beta_{1}^{*}(x_{R_{1}})=\beta_{2}^{*}(x_{R_{2}}).$\end{cor}
\begin{rem}
In the special case where the tree $T$ is linear we obtain the equivalent
results for simplicial sets. Namely, the presentation of the boundary
$\partial\Delta[n]$ and of the horn $\Lambda^{k}[n]$ as colimits
of standard simplices, and the description of a horn $\Lambda^{k}[n]\rightarrow X$
in a simplicial set $X$ (see \cite{Goers Jardin}).
\end{rem}
We are now able to define the dendroidal sets that model $\infty$-operads. 
\begin{defn}
A dendroidal set $X$ is an $\infty$-\emph{operad }if every inner
horn $h:\Lambda^{e}[T]\to X$ has a filler $\hat{h}:\Omega[T]\to X$
making the diagram\[
\xymatrix{\Lambda^{e}[T]\ar[d]\ar[r]^{h} & X\\
\Omega[T]\ar[ur]_{\hat{h}}}
\]
commute. 
\end{defn}
The following relation between $\infty$-categories and $\infty$-operads
is trivial to prove:
\begin{prop}
If $X$ is an $\infty$-category then $i_{!}(X)$ is an $\infty$-operad.
If $Y$ is an $\infty$-operad then $i^{*}(Y)$ is an $\infty$-category. 
\end{prop}
It is not hard to see that given any symmetric operad $\mathcal{P}$
its dendroidal nerve $N_{d}(\mathcal{P})$ is an $\infty$-operad.
In fact we can characterize those dendroidal sets occurring as nerves
of operads as follows. 
\begin{defn}
An $\infty$-operad $X$ is called \emph{strict }if any inner horn
in $X$ as above has a unique filler.\end{defn}
\begin{lem}
\label{lem:strict iff operad}A dendroidal set $X$ is a strict $\infty$-operad
if, and only if, there is an operad $\mathcal{P}$ such that $N_{d}(\mathcal{P})\cong X$. 
\end{lem}
A family of examples of paramount importance of $\infty$-operads
are given by the following. Recall that when $\mathcal{E}$ is a symmetric monoidal
model category a symmetric operad $\mathcal{P}$ enriched in $\mathcal{E}$
is called \emph{locally fibrant} if each hom-object in $\mathcal{P}$
is fibrant in $\mathcal{E}$. 
\begin{thm}
Let $\mathcal{P}$ be a locally fibrant symmetric operad in $\mathcal{E}$,
where $\mathcal{E}$ is a homotopy environment. The homotopy coherent
nerve $hcN_{d}(\mathcal{P})$ is an $\infty$-operad. 
\end{thm}

\subsection{The Cisinski-Moerdijk model category structure on $dSet$}

The objective of this section is to present the Cisinski-Moerdijk
model structure on $dSet$. All of the material in this section is
taken from \cite{dSet model hom op}, to which the reader is referred
to for more information and the proofs. In this model structure $\infty$-operads
are the fibrant objects and it is closely related to the operadic
model structure on $Ope$ and to the Joyal model structure on $sSet$. 

We note immediately that the Cisinski-Moerdijk model structure is
not a Cisinski model structure (i.e., a model structure on a presheaf
category such that the cofibrations are precisely the monomorphisms)
due to a technical complication that prevents the direct application
of the techniques developed in \cite{Cisinski his model stru}. Indeed,
the cofibrations in the model structure are the so-called \emph{normal
monomorphisms}.
\begin{defn}
A monomorphism of dendroidal sets $f:X\rightarrow Y$ is \emph{normal}
if for every dendrex $t\in Y_{T}$ that does not factor through $f$
the only isomorphism of $T$ that fixes $t$ is the identity. 
\end{defn}
An important property of dendroidal sets, proved in \cite{inn Kan in dSet},
is the following. 
\begin{thm}
\label{thm:exp}Let $X$ be a normal dendroidal set (i.e., $\emptyset\to X$
is normal) and $Y$ an $\infty$-operad. The dendroidal set $[X,Y]$
is again an $\infty$-operad. \end{thm}
\begin{proof}
The proof uses the technique of anodyne extensions, as commonly used
in the theory of simplicial sets (e.g., \cite{Gab Zisman,Goers Jardin}),
suitably adapted to dendroidal sets. Technically though, the dendroidal
case is much more difficult. For simplicial sets there is a rather
simple description of the non-degenerate simplices of $\Delta[n]\times\Delta[k]$.
But for trees $S$ and $T$ a similar description of the non-degenerate
dendrices of $\Omega[S]\otimes\Omega[T]$ is given by the so called
poset of percolation trees associated with $S$ and $T$. Complete
details can be found in \cite{inn Kan in dSet} and we just briefly
illustrate the construction for the trees $S$ and $T$: \[
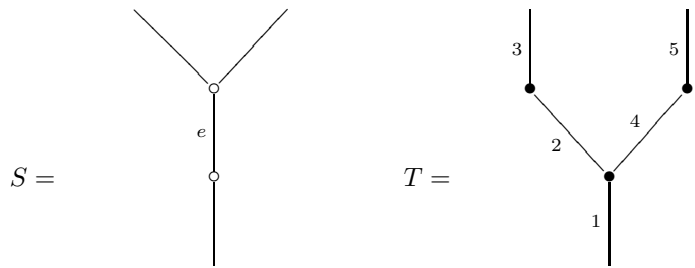
\begin{array}{ccc}
\xymatrix{ & *{\,}\ar@{-}[dr] &  & *{\,}\ar@{-}[dl]\\
 &  & *{\circ}\ar@{-}[d]_{e}\\
S= &  & *{\circ}\ar@{-}[d]\\
 &  & *{\,}}
 & \quad\quad & \xymatrix{ & *{\,}\ar@{-}[d]_{3} &  & *{\,}\ar@{-}[d]_{5}\\
 & *{\bullet}\ar@{-}[dr]_{2} &  & *{\bullet}\ar@{-}[dl]_{4}\\
T= &  & *{\bullet}\ar@{-}[d]_{1}\\
 &  & *{\,}}
\end{array}\]
on the following page.

\newpage

The presentation of $\Omega[T]\otimes\Omega[S]$ is given by the 14
trees $T_{1},\cdots,T_{14}$:

$\xyR{5pt}\xyC{5pt}
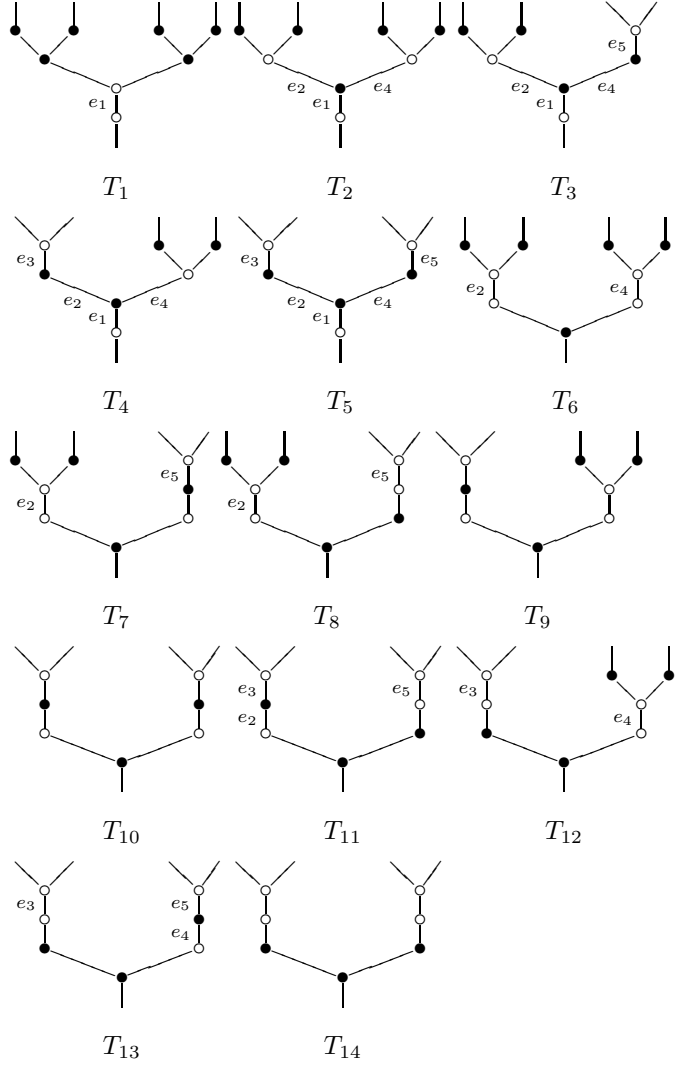
\xymatrix{*{}\ar@{-}[d] &  & *{}\ar@{-}[d] &  & *{}\ar@{-}[d] &  & *{}\ar@{-}[d]\\
*{\bullet}\ar@{-}[dr] &  & *{\bullet}\ar@{-}[dl] &  & *{\bullet}\ar@{-}[dr] &  & *{\bullet}\ar@{-}[dl]\\
 & *{\bullet}\ar@{-}[drr] &  &  &  & *{\bullet}\ar@{-}[dll]\\
 &  &  & *{\circ}\ar@{-}[d]_{e_{1}}\\
 &  &  & *{\circ}\ar@{-}[d]\\
 &  &  & *{}\\
 &  &  & T_{1}}
$~~~~~$\xyR{5pt}\xyC{5pt}\xymatrix{*{}\ar@{-}[d] &  & *{}\ar@{-}[d] &  & *{}\ar@{-}[d] &  & *{}\ar@{-}[d]\\
*{\bullet}\ar@{-}[dr] &  & *{\bullet}\ar@{-}[dl] &  & *{\bullet}\ar@{-}[dr] &  & *{\bullet}\ar@{-}[dl]\\
 & *{\circ}\ar@{-}[drr]_{e_{2}} &  &  &  & *{\circ}\ar@{-}[dll]^{e_{4}}\\
 &  &  & *{\bullet}\ar@{-}[d]_{e_{1}}\\
 &  &  & *{\circ}\ar@{-}[d]\\
 &  &  & *{}\ar@{-}[]\\
 &  &  & T_{2}}
$~~~~~$\xyR{5pt}\xyC{5pt}\xymatrix{*{}\ar@{-}[d] &  & *{}\ar@{-}[d] &  & *{}\ar@{-}[dr] &  & *{}\ar@{-}[dl]\\
*{\bullet}\ar@{-}[dr] &  & *{\bullet}\ar@{-}[dl] &  &  & *{\circ}\ar@{-}[d]_{e_{5}}\\
 & *{\circ}\ar@{-}[drr]_{e_{2}} &  &  &  & *{\bullet}\ar@{-}[dll]^{e_{4}}\\
 &  &  & *{\bullet}\ar@{-}[d]_{e_{1}}\\
 &  &  & *{\circ}\ar@{-}[d]\\
 &  &  & *{}\\
 &  &  & T_{3}}
$

$\xyR{5pt}\xyC{5pt}\xymatrix{*{}\ar@{-}[dr] &  & *{}\ar@{-}[dl] &  & *{}\ar@{-}[d] &  & *{}\ar@{-}[d]\\
 & *{\circ}\ar@{-}[d]_{e_{3}} &  &  & *{\bullet}\ar@{-}[dr] &  & *{\bullet}\ar@{-}[dl]\\
 & *{\bullet}\ar@{-}[drr]_{e_{2}} &  &  &  & *{\circ}\ar@{-}[dll]^{e_{4}}\\
 &  &  & *{\bullet}\ar@{-}[d]_{e_{1}}\\
 &  &  & *{\circ}\ar@{-}[d]\\
 &  &  & *{}\\
 &  &  & T_{4}}
$~~~~~$\xyR{5pt}\xyC{5pt}\xymatrix{*{}\ar@{-}[dr] &  & *{}\ar@{-}[dl] &  & *{}\ar@{-}[dr] &  & *{}\ar@{-}[dl]\\
 & *{\circ}\ar@{-}[d]_{e_{3}} &  &  &  & *{\circ}\ar@{-}[d]^{e_{5}}\\
 & *{\bullet}\ar@{-}[drr]_{e_{2}} &  &  &  & *{\bullet}\ar@{-}[dll]^{e_{4}}\\
 &  &  & *{\bullet}\ar@{-}[d]_{e_{1}}\\
 &  &  & *{\circ}\ar@{-}[d]\\
 &  &  & *{}\\
 &  &  & T_{5}}
$~~~~~$\xyR{5pt}\xyC{5pt}\xymatrix{*{}\ar@{-}[d] &  & *{}\ar@{-}[d] &  & *{}\ar@{-}[d] &  & *{}\ar@{-}[d]\\
*{\bullet}\ar@{-}[dr] &  & *{\bullet}\ar@{-}[dl] &  & *{\bullet}\ar@{-}[dr] &  & *{\bullet}\ar@{-}[dl]\\
 & *{\circ}\ar@{-}[d]_{e_{2}} &  &  &  & *{\circ}\ar@{-}[d]_{e_{4}}\\
 & *{\circ}\ar@{-}[drr] &  &  &  & *{\circ}\ar@{-}[dll]\\
 &  &  & *{\bullet}\ar@{-}[d]\\
 &  &  & *{}\\
 &  &  & T_{6}}
$

$\xyR{5pt}\xyC{5pt}\xymatrix{*{}\ar@{-}[d] &  & *{}\ar@{-}[d] &  & *{}\ar@{-}[dr] &  & *{}\ar@{-}[dl]\\
*{\bullet}\ar@{-}[dr] &  & *{\bullet}\ar@{-}[dl] &  &  & *{\circ}\ar@{-}[d]_{e_{5}}\\
 & *{\circ}\ar@{-}[d]_{e_{2}} &  &  &  & *{\bullet}\ar@{-}[d]\\
 & *{\circ}\ar@{-}[drr] &  &  &  & *{\circ}\ar@{-}[dll]\\
 &  &  & *{\bullet}\ar@{-}[d]\\
 &  &  & *{}\\
 &  &  & T_{7}}
$~~~~~$\xyR{5pt}\xyC{5pt}\xymatrix{*{}\ar@{-}[d] &  & *{}\ar@{-}[d] &  & *{}\ar@{-}[dr] &  & *{}\ar@{-}[dl]\\
*{\bullet}\ar@{-}[dr] &  & *{\bullet}\ar@{-}[dl] &  &  & *{\circ}\ar@{-}[d]_{e_{5}}\\
 & *{\circ}\ar@{-}[d]_{e_{2}} &  &  &  & *{\circ}\ar@{-}[d]\\
 & *{\circ}\ar@{-}[drr] &  &  &  & *{\bullet}\ar@{-}[dll]\\
 &  &  & *{\bullet}\ar@{-}[d]\\
 &  &  & *{}\\
 &  &  & T_{8}}
$~~~~~$\xyC{5pt}\xyR{5pt}\xymatrix{*{}\ar@{-}[dr] &  & *{}\ar@{-}[dl] &  & *{}\ar@{-}[d] &  & *{}\ar@{-}[d]\\
 & *{\circ}\ar@{-}[d] &  &  & *{\bullet}\ar@{-}[dr] &  & *{\bullet}\ar@{-}[dl]\\
 & *{\bullet}\ar@{-}[d] &  &  &  & *{\circ}\ar@{-}[d]\\
 & *{\circ}\ar@{-}[drr] &  &  &  & *{\circ}\ar@{-}[dll]\\
 &  &  & *{\bullet}\ar@{-}[d]\\
 &  &  & *{}\\
 &  &  & T_{9}}
$

$\xyC{5pt}\xyR{5pt}\xymatrix{*{}\ar@{-}[dr] &  & *{}\ar@{-}[dl] &  & *{}\ar@{-}[dr] &  & *{}\ar@{-}[dl]\\
 & *{\circ}\ar@{-}[d] &  &  &  & *{\circ}\ar@{-}[d]\\
 & *{\bullet}\ar@{-}[d] &  &  &  & *{\bullet}\ar@{-}[d]\\
 & *{\circ}\ar@{-}[drr] &  &  &  & *{\circ}\ar@{-}[dll]\\
 &  &  & *{\bullet}\ar@{-}[d]\\
 &  &  & *{}\\
 &  &  & T_{10}}
$~~~~~$\xyR{5pt}\xyC{5pt}\xymatrix{*{}\ar@{-}[dr] &  & *{}\ar@{-}[dl] &  & *{}\ar@{-}[dr] &  & *{}\ar@{-}[dl]\\
 & *{\circ}\ar@{-}[d]_{e_{3}} &  &  &  & *{\circ}\ar@{-}[d]_{e_{5}}\\
 & *{\bullet}\ar@{-}[d]_{e_{2}} &  &  &  & *{\circ}\ar@{-}[d]\\
 & *{\circ}\ar@{-}[drr] &  &  &  & *{\bullet}\ar@{-}[dll]\\
 &  &  & *{\bullet}\ar@{-}[d]\\
 &  &  & *{}\\
 &  &  & T_{11}}
$~~~~~$\xyC{5pt}\xyR{5pt}\xymatrix{*{}\ar@{-}[dr] &  & *{}\ar@{-}[dl] &  & *{}\ar@{-}[d] &  & *{}\ar@{-}[d]\\
 & *{\circ}\ar@{-}[d]_{e_{3}} &  &  & *{\bullet}\ar@{-}[dr] &  & *{\bullet}\ar@{-}[dl]\\
 & *{\circ}\ar@{-}[d] &  &  &  & *{\circ}\ar@{-}[d]_{e_{4}}\\
 & *{\bullet}\ar@{-}[drr] &  &  &  & *{\circ}\ar@{-}[dll]\\
 &  &  & *{\bullet}\ar@{-}[d]\\
 &  &  & *{}\\
 &  &  & T_{12}}
$

$\xyC{5pt}\xyR{5pt}\xymatrix{*{}\ar@{-}[dr] &  & *{}\ar@{-}[dl] &  & *{}\ar@{-}[dr] &  & *{}\ar@{-}[dl]\\
 & *{\circ}\ar@{-}[d]_{e_{3}} &  &  &  & *{\circ}\ar@{-}[d]_{e_{5}}\\
 & *{\circ}\ar@{-}[d] &  &  &  & *{\bullet}\ar@{-}[d]_{e_{4}}\\
 & *{\bullet}\ar@{-}[drr] &  &  &  & *{\circ}\ar@{-}[dll]\\
 &  &  & *{\bullet}\ar@{-}[d]\\
 &  &  & *{}\\
 &  &  & T_{13}}
$~~~~~$\xyC{5pt}\xyR{5pt}\xymatrix{*{}\ar@{-}[dr] &  & *{}\ar@{-}[dl] &  & *{}\ar@{-}[dr] &  & *{}\ar@{-}[dl]\\
 & *{\circ}\ar@{-}[d] &  &  &  & *{\circ}\ar@{-}[d]\\
 & *{\circ}\ar@{-}[d] &  &  &  & *{\circ}\ar@{-}[d]\\
 & *{\bullet}\ar@{-}[drr] &  &  &  & *{\bullet}\ar@{-}[dll]\\
 &  &  & *{\bullet}\ar@{-}[d]\\
 &  &  & *{}\\
 &  &  & T_{14}}
$ \\

The poset structure on these trees is \[
\xyC{10pt}\xyR{10pt}
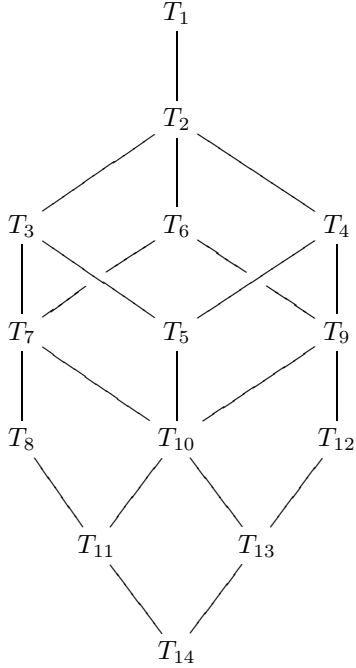
\xymatrix{ &  & T_{1}\ar@{-}[dd]\\
\\ &  & T_{2}\ar@{-}[ddrr]\ar@{-}[dd]\ar@{-}[ddll]\\
\\T_{3}\ar@{-}[dd]\ar@{-}[ddrr] &  & T_{6}\ar@{-}'[dl][ddll]\ar@{-}'[dr][ddrr] &  & T_{4}\ar@{-}[dd]\ar@{-}[ddll]\\
 & \, &  & \,\\
T_{7}\ar@{-}[ddrr]\ar@{-}[dd] &  & T_{5}\ar@{-}[dd] &  & T_{9}\ar@{-}[ddll]\ar@{-}[dd]\\
 & \,\\
T_{8}\ar@{-}[ddr] &  & T_{10}\ar@{-}[ddr]\ar@{-}[ddl] &  & T_{12}\ar@{-}[ddl]\\
\\ & T_{11}\ar@{-}[ddr] &  & T_{13}\ar@{-}[ddl]\\
\\ &  & T_{14}}
\]

\end{proof}
As a special case of this result we may now recover Boardman and Vogt's
result that the simplicial set $wAlg[\mathcal{P},\mathcal{E}]$ of
weak $\mathcal{P}$-algebras in $\mathcal{E}$ is an $\infty$-category,
as follows.
\begin{thm}
Let $\mathcal{P}$ be a symmetric operad and $\mathcal{E}$ a homotopy
environment. If the dendroidal set $N_{d}(\mathcal{P})$ is normal
and $hcN_{d}(\hat{\mathcal{E}})$ is an $\infty$-operad then $wAlg[\mathcal{P},\mathcal{E}]$
is an $\infty$-category.\end{thm}
\begin{proof}
The proof follows by noticing that $i^{*}([N_{d}(\mathcal{P}),hcN_{d}(\hat{\mathcal{E}})])\cong wAlg[\mathcal{P},\mathcal{E}]$.
\end{proof}
As we have seen above, local fibrancy of $\hat{\mathcal{E}}$ assures
that $hcN_{d}(\hat{\mathcal{E}})$ is an $\infty$-operad. See below
for a condition on $\mathcal{P}$ sufficient to assure $N_{d}(\mathcal{P})$
is normal. We now turn to the Cisinski-Moerdijk model structure. 
\begin{thm}
The category $dSet$ of dendroidal sets admits a Quillen model structure
where the cofibrations are the normal monomorphisms, the fibrant objects
are the $\infty$-operads, and the fibrations between $\infty$-operads
are the inner Kan fibrations whose image under $\tau_{d}$ is an operadic
fibration. The class $\mathcal{W}$ of weak equivalences can be characterized
as the smallest class of arrows which contains all inner anodyne extensions,
all trivial fibrations between $\infty$-operads and satisfies the
2 out of 3 property. Furthermore, with the tensor product of dendroidal
sets, this model structure is a monoidal model category. Slicing this
model structure recovers the Joyal model structure on $sSet$ and
in the diagram\[
\xyC{25pt}\xyR{25pt}\xymatrix{Cat\ar@<2pt>[r]^{j_{!}\,\,\,}\ar@<2pt>[d]^{N} & Ope\ar@<2pt>[l]^{j^{*}\,\,\,}\ar@<2pt>[d]^{N_{d}}\\
sSet\ar@<2pt>[r]^{i_{!}}\ar@<2pt>[u]^{\tau} & dSet\ar@<2pt>[l]^{i^{*}}\ar@<2pt>[u]^{\tau_{d}}}
\]
where the categories are endowed (respectively starting from the top-left
going clock-wise) with the categorical, operadic, Cisinski-Moerdijk,
and Joyal model structures all adjunctions are Quillen adjunction
(and none is a Quillen equivalence). \end{thm}
\begin{proof}
The proof of the model structure is quite intricate and is established
in \cite{dSet model hom op}. 
\end{proof}

\subsection{Homotopy invariance property for algebras in $dSet$}

As a consequence of the Cisinski-Moerdijk model structure on dendroidal
sets we obtain the following. 
\begin{thm}
Let $X$ be a normal dendroidal set and $E$ an $\infty$-operad.
Then $X$-algebras in $E$ have the homotopy invariance property. \end{thm}
\begin{proof}
In the diagram defining the homotopy invariance property in Definition
\ref{def:homotop inv proper} above the left vertical arrow is a trivial
cofibration and the right vertical arrow is a fibration in the Cisinski-Moerdijk
model structure and thus the required lift exists. 
\end{proof}
We now recall Fact \ref{fac:iso invar} regarding the internalization
of strict algebras by the closed monoidal structure on $Ope$ given
by the Boardman-Vogt tensor product and the isomorphism invariance
property for such algebras as captured by the operadic monoidal model
structure on $Ope$. We recall that a similar such correspondence
for weak algebras and their homotopy invariance property does not
seem possible within the confines of enriched operads. We are now
in a position to summarize the results recounted above in a form completely
analogous to the situation of strict algebras. 
\begin{fact}
The notion of algebras of dendroidal sets is internalized to the category
$dSet$ by it being closed monoidal with respect to the tensor product
of dendroidal sets. The homotopy invariance property of $X$-algebras
in an $\infty$-operad, for a normal $X$, holds and is captured by
the fact that $dSet$ supports the Cisinski-Moerdijk model structure
which is compatible with the tensor product. The notion of a weak
$\mathcal{P}$-algebra in a homotopy environment $\mathcal{E}$, where
$\mathcal{P}$ is discrete, is subsumed by the notion of algebras
in $dSet$ by means of the dendroidal set $[N_{d}(\mathcal{P}),hcN_{d}(\mathcal{E})]$
of weak $\mathcal{P}$-algebras in $\mathcal{E}$. 
\end{fact}

\subsection{\label{sub:Revisiting-applications}Revisiting applications}

Recall the iterative construction of the dendroidal set $nA_{\infty}$
of $n$-fold $A_{\infty}$-spaces and $_{n}Cat$ of weak $n$-categories
as well as $wCat_{k}^{n}$ of $k$-monoidal $n$-categories. It follows
from the general theory of dendroidal sets that these are all $\infty$-operads.
To see that, one uses Theorem \ref{thm:exp} together with the fact
that for a $\Sigma$-free symmetric operad $\mathcal{P}$ (for instance
if $\mathcal{P}$ is obtained from a planar operad $\mathcal{Q}$
by the symmetrization functor) then $N_{d}(\mathcal{P})$ is normal.
Thus, weak maps of $n$-fold $A_{\infty}$-spaces can be coherently
composed and similarly so can weak functors between weak $n$-categories.
As for $k$-monoidal $n$-categories we note that $wCat_{k}^{n}$
is, for similar reasons as above, an $\infty$-operad as well. 

We may now reduce the Baez-Breen-Dolan stabilization conjecture as
follows.
\begin{prop}
If for any $n\ge0$ the dendroidal set $wCat_{n}^{n}$ is a strict
$\infty$-operad then the Baez-Breen-Dolan stabilization conjecture
is true. \end{prop}
\begin{proof}
Recall that $wCat_{n}^{n}=[A^{\otimes n},wCat^{n}]$ and assume it
is a strict $\infty$-operad. We wish to prove that $[A^{\otimes n+j},wCat^{n}]\cong[A^{\otimes n+2},wCat^{n}]$
for any fixed $j>2$. By Lemma \ref{lem:strict iff operad} there
is an operad $\mathcal{P}$ such that $[A^{\otimes n},wCat^{n}]=N_{d}(\mathcal{P})$.
We now have\[
[A^{\otimes n+j},wCat^{n}]=[A^{\otimes j},[A^{\otimes n},wCat^{n}]]=[A^{\otimes j},N_{d}(\mathcal{P})]\]
which by adjunction is isomorphic to $N_{d}([\tau_{d}(A^{\otimes j}),\mathcal{P}]).$
However, $A$ is actually the dendroidal nerve of the symmetric operad
$As$ classifying associative monoids. By Proposition \ref{pro:tau of tensor}
we have\[
\tau_{d}(A^{\otimes j})=\tau_{d}(N_{d}(As)^{\otimes j})\cong As^{\otimes j}\cong Comm.\]
and the result follows. 
\end{proof}

\section{Dendroidal sets - combinatorial models of unknown spaces}

All of the theory of dendroidal sets that directly or indirectly is
concerned with algebras (we include the Cisinski-Moerdijk model structure
here as well) is very operadic in nature and is closely related to
the theory of $\infty$-categories modeled by quasicategories. Simplicial
sets are, however, also models for topological homotopy theory. Indeed,
simplicial sets were introduced in the context of algebraic topology
as combinatorial models of topological spaces. The appropriate equivalence
is established in \cite{Quillen hom alg} and was the reason to introduce
Quillen model categories. To recall the main result recall the singular
functor $Sing:Top\to sSet$ and its left adjoint $|-|:sSet\to Top$
given by geometric realization. 
\begin{thm}
The category $Top$ supports a Quillen model structure in which the
weak equivalences are the weak homotopy equivalences and the fibrations
are the Serre fibrations. The category $sSet$ supports a Quillen
model structure in which the weak equivalences are those maps $f:X\to Y$
for which the geometric realization $|f|$ is a homotopy weak equivalence
and the fibrations are the Kan fibrations. With these model structures
the adjunction above is a Quillen equivalence. 
\end{thm}
Simplicial sets thus support a topologically flavoured Quillen model
structure as well as the Joyal model structure which is categorically
flavoured and so simplicial sets play two rather different roles.
The close connection between dendroidal sets and simplicial sets raises
the question as to the existence of a topologically flavoured interpretation
of dendroidal sets as well. This problem is open for debate and interpretation
and is certainly far from settled.

We remark first that there is some indication that suggests dendroidal
sets do carry topological meaning. Recall the Dold-Kan correspondence
that establishes an equivalence of categories between the category
$sAb$ of simplicial abelian groups and the category $Ch$ of non-negatively
graded chain complexes. This correspondence is useful in the calculation
of homotopy groups of simplicial sets and in the definition of Eilenberg-Mac
Lane spaces $K(G,n)$ for $n>1$. In \cite{den Dold-Kan} it is shown
that there is a planar dendroidal version (where one considers a planar
version of $\Omega$ whose objects are planar trees) of the Dold-Kan
correspondence. The equivalence is between the category $dAb$ of
planar dendroidal abelian groups and the category $dCh$ of planar
dendroidal chain complexes. The definition of the latter requires
that for each face map $\partial_{\alpha}$ between planar trees there
is associated a sign $sgn(\partial_{\alpha})\in\{\pm1\}$ such that
the following holds. In the planar version of $\Omega$ it is still
true that a face $S\to T$ of codimension $2$ decomposes in precisely
two ways as the composition of two faces (see Proposition \ref{pro:Sub-faceOfASub-Face}
above). Thus we can write $S\to T$ as $\partial_{\alpha}\circ\partial_{\beta}$
as well as $\partial_{\gamma}\circ\partial_{\delta}$ and we require
that $sgn(\partial_{\alpha})\cdot sgn(\partial_{\beta})=-sgn(\partial_{\gamma})\cdot sgn(\partial_{\delta})$.
One may now wonder whether these dendroidal chain complexes give rise
to some sort of generalized Eilenberg-Mac Lane spaces. A first step
towards answering this question should be a clearer specification
of goals in a broad context, which is the aim of the rest of this section.

As inspiration we consider the Quillen equivalence between topological
spaces and simplicial sets mentioned above. The geometric realization
plays there a prominent role and thus a significant aspect of understanding
the homotopy behind dendroidal sets is to find a category $dTop$ together
with functors $Sing_{d}:dTop\to dSet$ and $|-|_{d}:dSet\to dTop$. The
category $dTop$ of course has to be chosen with care so that it will
rightfully be considered to be related to topology. We thus expect
that there is a fully faithful functor $h_{!}:Top\to dTop$ with a right
adjoint $h^{*}:dTop\to Top$ that should be defined 'purely topologically'. We
thus expect $dTop$ to be a category of some generalized topological
spaces in which ordinary topological spaces embed via $h_{!}$. To
allow sufficient flexibility for working with these objects we expect
that $dTop$ be small complete and small cocomplete. Moreover, the functor
$|-|_{d}:dSet\to dTop$ should send a dendroidal set $X$ to some generalized
space $|X|_{d}$ in such a way that the combinatorial information
in $X$ is not lost. We thus expect of any such functor $|-|_{d}$
that if for some $f:X\to Y$ the map $|f|_{d}$ is an isomorphism
then $f$ was already an isomorphism. In other words we expect $|-|_{d}$
to be conservative. 

The term 'purely topologically' above is of course vague and open
to discussion. In an attempt to formalize it recall the various slicing
lemmas we have seen above: Slicing symmetric operads over $\star$
gives categories, slicing dendroidal sets over $\Omega[\eta]=N_{d}(\star)$
gives simplicial sets, and slicing $\Omega$ over $\eta$ gives $\Delta$.
We thus expect that there is an object $\star\in dTop$ such that slicing
$dTop$ over $\star$ gives a category equivalent to $Top$ and that in fact the embedding $h_{!}:Top\to dTop$
is essentially the forgetful functor $dTop/\star\to dTop$. Moreover, noting
that the `correct' tensor product of dendroidal sets is not the cartesian
one we expect $dTop$ to posses a monoidal structure different from the
cartesian product. And, just as the tensor product of dendroidal
sets slices to the cartesian product of simplicial sets we expect
the monoidal structure on $dTop$ to slice to the cartesian product of
topological spaces. Lastly, an important property of the ordinary
geometric realization functor is that it commutes with finite products. We
expect of the dendroidal geometric realization functor $dSet\to dTop$
to be monoidal with respect to the non-cartesian monoidal structure on each category. 

We summarize our expectations in the following formulation. 
\begin{problem}
Find a category $dTop$ together with a functor $Sing_{d}:dTop\to dSet$,
a left adjoint $|-|_{d}:dSet\to dTop$, and an object $\star\in dTop_{0}$
such that:\end{problem}
\begin{enumerate}
\item $dTop$ is small complete and small cocomplete. 
\item (Slicing lemma) $dTop/\star$ is equivalent to $Top$.
\item The forgetful functor $h_{!}:Top\to dTop$ is an embedding. 
\item Slicing $Sing_{d}$ gives $Sing$ and slicing $|-|_{d}$ gives $|-|$. 
\item $|-|_{d}$ is conservative. 
\item $dTop$ admits a non-cartesian monoidal structure that slices over $\star$
to the cartesian product in $Top$ (along $h_{!}$).
\item The functor $|-|_{d}$ is to be a monoidal functor with respect to
the tensor structures on $dSet$ and $dTop$.
\end{enumerate}
We would thus obtain the diagram\[
\xymatrix{sSet\ar@<2pt>[r]^{|-|}\ar@<-2pt>[d]_{i_{!}} & Top\ar@<2pt>[l]^{Sing}\ar@<-2pt>[d]_{h_{!}}\\
dSet\ar@<-2pt>[u]_{i^{*}}\ar@<2pt>[r]^{|-|_{d}} & dTop\ar@<2pt>[l]^{Sing_{d}}\ar@<-2pt>[u]_{h^{*}}}
\]
where both squares commute. 

The quest will be complete with the establishment of Quillen model
structures on $dSet$ and $dTop$ that slice respectively to the standard
(topological) ones on $sSet$ and $Top$ and such that in the square
above all adjunctions are Quillen adjunctions with both horizontal
ones Quillen equivalences.

\end{document}